\documentclass{article}
\usepackage{amsfonts,amsmath,latexsym,hyperref} % label
\newcounter{cornum}
\newtheorem{corollary}{Corollary}[cornum]

\newcounter{defnum}
\newtheorem{definition}{Definition}[defnum]

\newcounter{examplenum}
\newtheorem{example}{Example}[examplenum]

\newcounter{lemmanum}
\newtheorem{lemma}{Lemma}[lemmanum]

\newcounter{satznum}
\newtheorem{theorem}{Theorem}[satznum]

\newenvironment{acknowledgement}
 {\begin{trivlist}\item[]{\bf Acknowledgement.}}
 {\end{trivlist}}
\newenvironment{remark}
 {\begin{trivlist}\item[]{\bf Remark.}}
 {\end{trivlist}}

\newenvironment{proof}
 {\begin{trivlist}\item[]{\bf Proof.}}
 {\end{trivlist}}
{\makeatletter
\gdef\me{{\mathbb E}} % expectation
\gdef\nz{{\mathbb N}} % positive integers
\gdef\pr{{\mathbb P}} % probability
\gdef\rz{{\mathbb R}} % real numbers
}
\newcounter{todocounter}

\makeatletter
\def\@MRExtract#1 #2!{#1}
\newcommand{\MR}[1]{% we need to strip the "(...)"
  \xdef\@MRSTRIP{\@MRExtract#1 !}
  \href{http://www.ams.org/mathscinet-getitem?mr=\@MRSTRIP}{MR\@MRSTRIP}}
\makeatother

\begin{document}
   \section*{Asymptotic genealogies for a class of generalized
   Wright--Fisher models}
\begin{center}
%   Incomplete draft version\\
%   (not ready for submission)\\
   Date: \today
\end{center}
   {\sc Thierry Huillet}\footnote{Laboratoire de Physique Th\'eorique
   et Mod\'elisation, CY Cergy Paris Universit\'e, CNRS UMR 8089, 2
   avenue Adolphe-Chauvin, 95302 Cergy-Pontoise, France, e-mail: thierry.huillet@cyu.fr} and
   {\sc Martin M\"ohle}\footnote{Mathematisches Institut, Eberhard Karls Universit\"at T\"ubingen,
   Auf der Morgenstelle 10, 72076 T\"ubingen, Germany, e-mail: martin.moehle@uni-tuebingen.de}
\begin{abstract}
   We study a class of Cannings models with population size $N$ having
   a mixed multinomial offspring distribution with random success
   probabilities $W_1,\ldots,W_N$ induced by independent and identically
   distributed positive random variables $X_1,X_2,\ldots$ via
   $W_i:=X_i/S_N$, $i\in\{1,\ldots,N\}$, where $S_N:=X_1+\cdots+X_N$.
   The ancestral lineages are hence based on a sampling with replacement
   strategy from a random partition of the unit interval into $N$
   subintervals of lengths $W_1,\ldots,W_N$. Convergence results for the
   genealogy of these Cannings models are provided under regularly varying
   assumptions on the tail distribution of $X_1$. In the limit several
   coalescent processes with multiple and simultaneous multiple collisions
   occur. The results extend those obtained in \cite{huillet} for the case
   when $X_1$ is Pareto distributed and complement those obtained in
   \cite{schweinsberg3} for models where one samples without replacement
   from a supercritical branching process.

   \vspace{2mm}

   \noindent Keywords: Cannings model; exchangeable coalescent; regularly
   varying function; simultaneous multiple collisions; weak convergence

   \vspace{2mm}

   \noindent 2020 Mathematics Subject Classification:
            Primary 60J90 % Coalescent processes
            %; xxxxx
            Secondary 92D15 % Problems related to evolution
            %; xxxxx

   \vspace{2mm}

   \noindent Running title: Asymptotic genealogies
\end{abstract}
\subsection{Introduction} \label{intro}
   Let $X_1,X_2,\ldots$ be independent copies of a
   random variable $X$ taking values in $(0,\infty)$.
   For $N\in\nz:=\{1,2,\ldots\}$ define $S_N:=X_1+\cdots+X_N$ and
   $W_i:=X_i/S_N$, $i\in\{1,\ldots,N\}$. The weights $W_1,\ldots,W_N$
   are exchangeable random variables with $W_1+\cdots+W_N=1$.
   In particular, $\me(W_i)=1/N$, $i\in\{1,\ldots,N\}$. Consider the
   Cannings model \cite{cannings1, cannings2} with population size $N$
   and non-overlapping generations such that, conditional on
   $W_1,\ldots,W_N$, the offspring sizes $\nu_1,\ldots,\nu_N$ have a
   multinomial distribution with parameters $N$ and $W_1,\ldots,W_N$.
   Thus, the offspring distribution is
   \begin{equation} \label{offspringdist}
      \pr(\nu_1=i_1,\ldots,\nu_N=i_N)
      \ =\ \frac{N!}{i_1!\cdots i_N!}\me(W_1^{i_1}\cdots W_N^{i_N}),
   \end{equation}
   $i_1,\ldots,i_N\in\nz_0:=\{0,1,2,\ldots\}$ with $i_1+\cdots+i_N=N$. For
   degenerate $X$, i.e. $\pr(X=c)=1$ for some real constant $c>0$, this
   model reduces to the classical Wright--Fisher model with deterministic
   weights $W_i=1/N$, $i\in\{1,\ldots,N\}$. It is straightforward to check
   that the offspring sizes have joint descending factorial moments
   \begin{equation}
      \me((\nu_1)_{k_1}\cdots(\nu_N)_{k_N})
      \ =\ (N)_{k_1+\cdots+k_N}\me(W_1^{k_1}\cdots W_N^{k_N}),
      \qquad k_1,\ldots,k_N\in\nz_0,
   \end{equation}
   where $(x)_0:=1$ and $(x)_k:=x(x-1)\cdots(x-k+1)$ for $x\in\rz$ and
   $k\in\nz$. In \cite{huillet} this model is studied for the case when
   $X$ is Pareto distributed. If $X$ is gamma distributed with density
   $x\mapsto x^{r-1}e^{-x}/\Gamma(r)$, $x>0$, for some $r>0$, then
   $(W_1,\ldots,W_N)$ is symmetric Dirichlet distributed with parameter
   $r$, leading to the Cannings model with offspring distribution
   \[
   \pr(\nu_1=i_1,\ldots,\nu_N=i_N)\ =\ \frac{N!}{i_1!\cdots i_N!}
   \frac{[r]_{i_1}\cdots [r]_{i_N}}{[rN]_N},
   \]
   $i_1,\ldots,i_N\in\nz_0$ with $i_1+\cdots+i_N=N$, where $[x]_0:=1$ and
   $[x]_i:=x(x+1)\cdots(x+i-1)$ for $x\in\rz$ and $i\in\nz$. This Dirichlet
   multinomial model has been studied extensively in the literature
   (see, for example, Griffiths and Span\`o \cite{griffithsspano}).
   In a series of papers
   \cite{huilletmoehlepop,huilletmoehlecor,huilletmoehleasy} a subclass of
   Cannings models, called conditional branching process models in
   the spirit of Karlin and McGregor \cite{karlinmcgregor1, karlinmcgregor2},
   has been investigated, whose offspring distributions are (by definition)
   obtained by assuming that $\pr(X_1+\cdots+X_N=N)>0$ and
   conditioning on the event that $X_1+\cdots+X_N=N$. This
   construction based on conditioning is rather different from the
   construction based on sampling from a random partition of the unit
   interval we are dealing with in this article. Note however that several
   concrete examples (such as the classical Wright--Fisher model and the
   above mentioned Dirichlet multinomial model) can be constructed in both
   ways, either by sampling or by conditioning.
   For example, the Dirichlet multinomial model is obtained by
   taking $N$ independent and identically distributed negative binomial
   random variables $X_1,\ldots,X_N$ with parameter $r>0$ and $p\in (0,1)$,
   so with distribution $\pr(X_1=k)=\binom{r+k-1}{k}p^r(1-p)^k$, $k\in\nz_0$,
   and conditioning on the event that $X_1+\cdots+X_N=N$.

   The closely related model studied by Schweinsberg \cite{schweinsberg3}
   differs from ours, since sampling is performed without replacement from
   a discrete super-critical Galton--Watson branching process, as explained
   in \cite[Section 1.3]{schweinsberg3}. In that model, $X$ is integer valued
   and satisfies $\me(X)>1$. In our model, $X$ does not need to be integer
   valued and its mean is allowed to be less than $1$. Moreover, the sampling
   in our multinomial model is with replacement, whereas in Schweinsberg's
   model it is without replacement.

   The same multinomial scheme with an additional dormancy mechanism is considered in the recent work of Cordero et al. \cite{cordero}. A class
   of Dirichlet models in the domain of attraction of the Kingman coalescent
   is also studied in two recent works of Boenkost et al. \cite{boenkost1,boenkost2} with an emphasis on Haldane's formula
   \cite{haldane}. We refer the reader to Athreya \cite{athreya} for
   some more information on Haldane's formula.

   Fix $n\in\{1,\ldots,N\}$ and sample $n$ individuals from the current
   generation. For $r\in\nz_0$ define a random partition
   $\Pi_r^{(N,n)}$ of $\{1,\ldots,n\}$ such that
   $i,j\in\{1,\ldots,n\}$ belong to the same block of $\Pi_r^{(N,n)}$
   if and only if the individual $i$ and $j$ share a common parent
   $r$ generations backward in time. The process
   $\Pi^{(N,n)}:=(\Pi^{(N,n)}_r)_{r\in\nz_0}$, called the discrete-time
   $n$-coalescent, takes values in the space ${\cal P}_n$ of partitions
   of $\{1,\ldots,n\}$. As in \cite{huillet} we are interested in the
   limiting behaviour of the discrete-time $n$-coalescent as the total
   population size $N$ tends to infinity. It is easily seen (and
   well-known) that the discrete-time $n$-coalescent is a time-homogeneous
   Markovian process. The transition probabilities
   $p_{\pi\pi'}:=\pr(\Pi_{r+1}^{(N,n)}=\pi'\,|\,\Pi_r^{(N,n)}=\pi)$
   are given by
   \begin{equation} \label{transprob}
      p_{\pi\pi'}\ =\ (N)_j \me(W_1^{k_1}\cdots W_j^{k_j})
      \ =:\ \Phi_j^{(N)}(k_1,\ldots,k_j),
      \qquad \pi,\pi'\in{\cal P}_n,
   \end{equation}
   if each block of $\pi'$ is a union of some blocks of $\pi$, where
   $j:=|\pi'|$ denotes the number of blocks of $\pi'$ and $k_1,\ldots,k_j$
   are the group sizes of merging blocks of $\pi$.
   Note that $\Phi_j^{(N)}(k_1,\ldots,k_j)$ is defined for all
   $N,j,k_1,\ldots,k_j\in\nz$. Since the random variables $W_1,\ldots,W_N$
   are exchangeable and satisfy $W_1+\cdots+W_N=1$, it follows for all
   $N,j,k_1,\ldots,k_j\in\nz$ with $j\le N$ that
   \begin{eqnarray*}
      &   & \hspace{-15mm}(N-j)\me(W_1^{k_1}\cdots W_j^{k_j}W_{j+1})
      \ = \ \me\big(W_1^{k_1}\cdots W_j^{k_j}(W_{j+1}+\cdots+W_N)\big)\\
      & = & \me\big(W_1^{k_1}\cdots W_j^{k_j}(1-W_1-\cdots-W_j)\big)\\
      & = & \me(W_1^{k_1}\cdots W_j^{k_j}) - \sum_{i=1}^j
            \me(W_1^{k_1}\cdots W_{i-1}^{k_{i-1}}W_i^{k_i+1}W_{i+1}^{k_{i+1}}\cdots W_j^{k_j}).
   \end{eqnarray*}
   Multiplication by $(N)_j$ ($=0$ for $j>N$) shows that the consistency
   relation
   \begin{equation} \label{consistent}
      \Phi_j^{(N)}(k_1,\ldots,k_j)
      \ =\ \Phi_{j+1}^{(N)}(k_1,\ldots,k_j,1)+\sum_{i=1}^j
           \Phi_j^{(N)}(k_1,\ldots,k_{i-1},k_i+1,k_{i+1},\ldots,k_j)
   \end{equation}
   holds for all $N,j,k_1,\ldots,k_j\in\nz$. Moreover,
   for all $j,l\in\nz$ with $j\ge l$ and all $k_1,\ldots,k_j,m_1,\ldots,m_l\in\nz$
   with $k_1\ge m_1,\ldots,k_l\ge m_l$, the monotonicity relation
   \begin{equation} \label{monotone}
      \Phi_j^{(N)}(k_1,\ldots,k_j)\ \le\ \Phi_l^{(N)}(m_1,\ldots,m_l)
   \end{equation}
   holds. Note that (\ref{monotone}) follows from (\ref{consistent}) by
   induction on the difference $d:=j-l\in\nz_0$. We refer the reader to
   \cite[Definition 2.2]{moehlesampling} and the remark thereafter for
   similar statements and proofs for the full class of Cannings models.
   Choosing $j=1$ and $k_1=2$ in (\ref{transprob}) shows
   that two individuals share a common ancestor one
   generation backward in time with probability $c_N:=\Phi_1^{(N)}(2)=N\me(W_1^2)$,
   the so-called coalescence probability. We also introduce the effective
   population size $N_e:=1/c_N$. Note that
   $c_N=N\me(W_1^2)\ge N(\me(W_1))^2=1/N$ or, equivalently,
   $N_e\le N$. All Cannings models having an effective population size
   strictly larger than $N$ (such as the Moran model having effective
   population size $N_e=N(N-1)/2>N$ for $N\ge 4$
   and most of the extended Moran models studied by
   Eldon and Wakeley \cite{eldonwakeley} and Huillet and M\"ohle
   \cite{huilletmoehleextended}) therefore do not belong
   to the class of models we are dealing with in this article.

   General results for Cannings models concerning the convergence of their
   genealogical tree to an exchangeable coalescent process as the total
   population size tends to infinity are provided in \cite{moehlesagitov}.
   For information on the theory of exchangeable coalescent processes we
   refer the reader to Pitman \cite{pitman}, Sagitov \cite{sagitov} and
   Schweinsberg \cite{schweinsberg1,schweinsberg2}. Coalescents with
   multiple collisions ($\Lambda$-coalescents) are Markovian stochastic
   processes taking values in the set of partitions of $\nz$. They are
   characterized by a finite measure $\Lambda$ on the unit interval.
   Important examples are Dirac-coalescents, where $\Lambda=\delta_a$
   is the Dirac measure at a given point $a\in [0,1]$, including the
   prominent Kingman coalescent (Kingman \cite{kingman1,kingman2,kingman3}),
   where $\Lambda=\delta_0$ is the Dirac measure at $0$, and the
   star-shaped coalescent, where $\Lambda=\delta_1$. Other important
   examples are beta coalescents, where $\Lambda=\beta(a,b)$ is the
   beta distribution with parameters $a,b>0$, including the
   Bolthausen--Sznitman coalescent, where $\Lambda$ is the uniform
   distribution on the unit interval ($a=b=1$).

   The full class of exchangeable coalescent processes ($\Xi$-coalescents)
   allowing for simultaneous multiple collisions of ancestral lineages
   is characterized by a finite measure $\Xi$ on the infinite
   simplex $\Delta:=\{x=(x_1,x_2,\ldots):x_1\ge x_2\ge\cdots\ge 0,
   \sum_{i=1}^\infty x_i\le 1\}$. An example is the two-parameter
   Poisson--Dirichlet coalescent with parameters $\alpha>0$ and
   $\theta>-\alpha$, where the characterizing measure
   $\nu({\rm d}x):=\Xi({\rm d}x)/\sum_{i=1}^\infty x_i^2$ on $\Delta$ is
   (by definition) the Poisson--Dirichlet distribution
   $\nu={\rm PD}(\alpha,\theta)$ with parameters $\alpha>0$ and
   $\theta>-\alpha$. For more information on the Poisson--Dirichlet coalescent
   we refer the reader to Section 6 of \cite{moehleproper}.
   In most studies, continuous-time coalescent processes $(\Pi_t)_{t\in T}$
   with index set $T=[0,\infty)$ are considered. Note however that all
   $\Xi$-coalescents can as well be introduced with discrete
   time $T=\nz_0$. In this case one speaks about a discrete-time
   $\Xi$-coalescent $(\Pi_r)_{r\in\nz_0}$.
   The following terminology is taken from \cite[Definition 2.1]{huilletmoehlepop}.
\begin{definition} \label{domain}
   (i) A Cannings model is said to be in the domain of attraction of a
   continuous-time coalescent $\Pi=(\Pi_t)_{t\ge 0}$ if for each sample size
   $n\in\nz$ the time-scaled ancestral process
   $(\Pi^{(N,n)}_{\lfloor t/c_N\rfloor})_{t\ge 0}$ converges in
   $D_{{\cal P}_n}([0,\infty))$ to $\Pi^{(n)}$ as $N\to\infty$, where
   $\Pi^{(n)}=(\Pi_t^{(n)})_{t\ge 0}$ denotes the restriction of $\Pi$
   to a sample of size $n$.

   (ii) Analogously, a Cannings model is said to be in the domain of
   attraction of a discrete-time coalescent $\Pi=(\Pi_r)_{r\in\nz_0}$
   if for each sample size $n\in\nz$ the ancestral process
   $(\Pi_r^{(N,n)})_{r\in\nz_0}$ converges in $D_{{\cal P}_n}(\nz_0)$
   to $\Pi^{(n)}$ as $N\to\infty$, where $\Pi^{(n)}=(\Pi_r^{(n)})_{r\in\nz_0}$
   denotes the restriction of $\Pi$ to a sample of size $n$.
\end{definition}
Conditions on the tails of the distribution of $X$ are provided which ensure
that the population model with offspring distribution (\ref{offspringdist})
is in the domain of attraction of some exchangeable coalescent process. The
tail condition is of the standard form $\pr(X>x)\sim x^{-\alpha}\ell(x)$
as $x\to\infty$, where $\alpha\ge 0$ and $\ell$ is a function slowly varying
at $\infty$. The results are collected in Theorem \ref{main} in Section
\ref{results}. It turns out that the three parameter values $\alpha\in\{0,1,2\}$
are boundary cases. Consequently, six different regimes ($\alpha>2$, $\alpha=2$,
$\alpha\in(1,2)$, $\alpha=1$, $\alpha\in(0,1)$ and $\alpha=0$) are considered
leading to different limiting behaviours of the ancestral process. Theorem
\ref{main} also provides the asymptotics of the coalescence probability $c_N$ as
$N\to\infty$ for all six cases. In Section \ref{examples} some
illustrating examples are provided including the case studied in
\cite{huillet} when $X$ is Pareto distributed. The proofs are provided in the
main Section \ref{proofs}. They are based on general convergence-to-the-coalescent
theorems for Cannings models provided in \cite{moehlesagitov}
and combine (Abelian and Tauberian) arguments from the theory of regularly
varying functions in the spirit of Karamata \cite{karamata1,karamata2,karamata3}
with techniques used by Huillet \cite{huillet} for the Pareto case and by
Schweinberg \cite{schweinsberg3} for the related model where the sampling
is performed without replacement.
\subsection{Results} \label{results}
For most of the results it is assumed that there exists a constant
$\alpha\ge 0$ and a function $\ell:(0,\infty)\to (0,\infty)$ slowly varying
at $\infty$ such that
\begin{equation} \label{maincond}
   \pr(X>x)\ \sim\ x^{-\alpha}\ell(x),\qquad x\to\infty.
\end{equation}
Our main result (Theorem
\ref{main}) clarifies the limiting behaviour of the ancestral
structure of the Cannings model with offspring distribution
(\ref{offspringdist}) as the total population size $N$ tends to infinity
under the assumption (\ref{maincond}). It turns out that the parameter values
$\alpha\in\{0,1,2\}$ are boundary cases. It is hence
natural to distinguish six regimes corresponding to
the parameter ranges $\alpha>2$, $\alpha=2$, $\alpha\in (1,2)$, $\alpha=1$,
$\alpha\in (0,1)$ and $\alpha=0$.
In order to state the result it is convenient to
introduce the function $\ell^*:(1,\infty)\to(0,\infty)$ via
\begin{equation} \label{lstar}
   \ell^*(x)\ :=\ \int_1^x \frac{\ell(t)}{t}\,{\rm d}t.
\end{equation}
Note that $\ell^*$ is non-decreasing, slowly varying at $\infty$ and satisfies
$\ell(x)/\ell^*(x)\to 0$ as $x\to\infty$, see for example Bingham and
Doney \cite[p.~717 and 718]{binghamdoney} or Eq.~(1.5.8) on p.~26 of
Bingham, Goldie and Teugels \cite{binghamgoldieteugels} and the remarks
thereafter. More precisely, for every $\lambda>0$, as $x\to\infty$
\begin{eqnarray*}
   \frac{\ell^*(\lambda x)-\ell^*(x)}{\ell(x)}
   & = &
   \frac{1}{\ell(x)}\int_x^{\lambda x} \frac{\ell(t)}{t}\,{\rm d}t
   \ = \ \int_1^\lambda \frac{\ell(xu)}{\ell(x)}\frac{1}{u}\,{\rm d}u
   \ \to\ \int_1^\lambda \frac{1}{u}\,{\rm d}u\ =\ \log\lambda,
\end{eqnarray*}
where the convergence holds by the uniform convergence theorem for slowly
varying functions. Thus, $\ell^*$ is a de Haan function (with
$\ell$-index $1$) and hence slowly varying. For general information on
de Haan theory we refer the reader to Chapter 3 of \cite{binghamgoldieteugels}.

The main (and only) result of this article is the following.
\begin{theorem} \label{main}
For the Cannings model with offspring distribution
(\ref{offspringdist}) the following assertions hold.
   \begin{enumerate}
      \item[(i)] If $\me(X^2)<\infty$ (in particular if (\ref{maincond})
         holds with $\alpha>2$) then the model is in the domain of
         attraction of the continuous-time Kingman coalescent and the
         coalescence probability $c_N$ satisfies $c_N\sim\rho/(\mu^2N)$
         as $N\to\infty$, where $\mu:=\me(X)$ and $\rho:=\me(X^2)$.
      \item[(ii)] If (\ref{maincond}) holds with
         $\alpha=2$ then the model is in the domain of attraction of
         the continuous-time Kingman coalescent and the coalescence
         probability $c_N$ satisfies $c_N\sim 2\ell^*(N)/(\mu^2N)$ as
         $N\to\infty$, where $\mu:=\me(X)$ and $\ell^*$ is defined
         via (\ref{lstar}).
      \item[(iii)] If (\ref{maincond}) holds with $\alpha\in(1,2)$ then
         the model is in the domain of attraction of the continuous-time
         $\Lambda$-coalescent with $\Lambda:=\beta(2-\alpha,\alpha)$ being
         the beta distribution with parameters $2-\alpha$ and $\alpha$.
         Moreover, the coalescence probability $c_N$ satisfies $c_N\sim
         \alpha{\rm B}(2-\alpha,\alpha)\mu^{-\alpha}\ell(N)/N^{\alpha-1}
         =\Gamma(2-\alpha)\Gamma(\alpha+1)\mu^{-\alpha}\ell(N)/N^{\alpha-1}$
         as $N\to\infty$, where $\mu:=\me(X)$.
      \item[(iv)] If (\ref{maincond}) holds with $\alpha=1$, then the
         model is in the domain of attraction of the continuous-time
         Bolthausen--Sznitman coalescent. If $(a_N)_{N\in\nz}$ is a sequence
         of positive real numbers satisfying $\ell^*(a_N)\sim a_N/N$ as
         $N\to\infty$, where $\ell^*$ is defined via (\ref{lstar}), then the
         coalescence probability $c_N$ satisfies
         $c_N\sim\ell(a_N)/\ell^*(a_N)\sim N\ell(a_N)/a_N$ as $N\to\infty$.
      \item[(v)] If (\ref{maincond}) holds with $\alpha\in(0,1)$, then
         the model is in the domain of attraction of the discrete-time
         $\Xi$-coalescent, where the characterising measure
         $\nu({\rm d}x):=\Xi({\rm d}x)/\sum_{i=1}^\infty x_i^2$ is the
         Poisson--Dirichlet distribution $\nu={\rm PD}(\alpha,0)$ with
         parameters $\alpha$ and $\theta:=0$. The coalescence probability
         satisfies $c_N\to 1-\alpha$ as $N\to\infty$.
      \item[(vi)] If (\ref{maincond}) holds with $\alpha=0$, then the model
         is in the domain of attraction of the discrete-time star-shaped
         coalescent and the coalescence probability satisfies
         $c_N\to 1$ as $N\to\infty$.
   \end{enumerate}
   In particular, for the first four cases (i) - (iv), $c_N\to 0$
   as $N\to\infty$.
\end{theorem}
The six cases of Theorem \ref{main} are summarized in
Table 1. In the table, $\mu:=\me(X)$, $\rho:=\me(X^2)$,
$\ell^*(x):=\int_1^x\ell(t)/t\,{\rm d}t$, $x>1$, and $(a_N)_{N\in\nz}$ is a
sequence such that $\ell^*(a_N)\sim a_N/N$ as $N\to\infty$.
\begin{center}
   \renewcommand{\arraystretch}{2}
   \begin{tabular}{ccc}
   \hline
   Condition        & Limiting coalescent & Coalescence probability\\
   \hline
   $\me(X^2)<\infty$ & Kingman
                     & $\displaystyle\sim\frac{\rho}{\mu^2 N}$\\
   $\alpha=2$        & Kingman
                     & $\displaystyle\sim \frac{2\ell^*(N)}{\mu^2N}$\\
   $1<\alpha<2$      & $\beta(2-\alpha,\alpha)$
                     & $\displaystyle\sim\frac{\Gamma(2-\alpha)\Gamma(\alpha+1)\ell(N)}{\mu^\alpha N^{\alpha-1}}$\\
   $\alpha=1$        & Bolthausen--Sznitman
                     & $\displaystyle\sim\frac{\ell(a_N)}{\ell^*(a_N)}\sim\frac{N\ell(a_N)}{a_N}$\\
   $\alpha\in (0,1)$ & discrete-time Poisson--Dirichlet($\alpha,0$)
                     & $\sim 1-\alpha$\\
   $\alpha=0$        & discrete-time star-shaped & $\sim 1$\\
   \hline
   \end{tabular}\\
   \ \\
   \ \\
   Table 1: Asymptotics of the ancestry of mixed multinomial Cannings models\\
   of the form (\ref{offspringdist})
   under the tail condition $\pr(X>x)\sim x^{-\alpha}\ell(x)$ as $x\to\infty$. \end{center}
\begin{remark}
   If $\ell(x)\equiv C$ for some constant $C>0$, then
   $\ell^*(x)=C\int_1^x t^{-1}\,{\rm d}t=C\log x$ as $x\to\infty$.
   Assume now in addition that $\alpha=1$. In this case, in part (iv)
   of Theorem \ref{main} one can choose $a_1:=1$ and $a_N:=CN\log N$, $N\in\nz\setminus\{1\}$. The
   coalescence probability thus satisfies $c_N\sim CN/a_N\sim 1/\log N$,
   in agreement with Proposition 6 of Huillet \cite{huillet} for the
   Pareto example $\pr(X>x)=1/x$, $x>1$. The same asymptotics for the
   coalescence probability holds for the related model considered by
   Schweinsberg (see \cite[Lemma 16]{schweinsberg3}) and, for example,
   when $X$ is discrete taking the value $k\in\nz$ with probability
   $\pr(X=k)=1/(k(k+1))$.
\end{remark}
\begin{remark}
   One may doubt that Theorem \ref{main} is valid when $X$ takes values
   close to $0$ with high probability such that $\me(1/S_N)=\infty$ for
   all $N\in\nz$. Typical examples of this form arise when the Laplace
   transform $\psi$ of $X$ satisfies $\psi(u)\sim L(u)$ as $u\to\infty$
   for some function $L$ slowly varying at $\infty$, or, equivalently
   (see Feller \cite{feller}, p.~445, Theorem 2 and p.~446, Theorem 3),
   if $\pr(X\le x)\sim L(1/x)$ as $x\to 0$. A concrete example is
   $P(X\le x)=1/(1-\log x)$, $0<x\le 1$. In this case, $L(x)=1/\log x$,
   $x>0$, and, hence, $\me(1/S_N)=\int_0^\infty(\psi(u))^N\,{\rm d}u=\infty$ for all $N\in\nz$. By Theorem \ref{main}
   this model is in the domain of attraction of the Kingman coalescent,
   since $\me(X^2)<\infty$.\\
   The finiteness or infiniteness of $\me(1/S_N)$ turns out to be irrelevant
   for the statements in Theorem \ref{main}, since the convergence results
   of Theorem \ref{main} solely depend on the limiting behaviour of the
   joint moments of the weights $W_1,\ldots,W_j$ as $N\to\infty$. For
   example (see Lemma \ref{fundamental}), the asymptotics of $\me(W_1^p)$,
   $p>0$, as $N\to\infty$ is determined by the values $\psi(u)$ of the Laplace
   transform $\psi$ for values of $u$ close to $0$. For any fixed $\delta>0$
   the values $u>\delta$ do not play any role.
\end{remark}

\textbf{Conjectures and open problems.}\\
% We briefly compare our results with related models and discuss
% some open problems.
Theorem \ref{main} should also hold for Schweinberg's model
\cite{schweinsberg3}, since sampling without replacement (instead of sampling
with replacement) should neither influence the asymptotics of the coalescence
probability nor the limiting processes arising in Theorem \ref{main}. Note
that in \cite{schweinsberg3} the subclass of models without replacement is
studied where the function $\ell$ in (\ref{maincond}) is constant.
We leave the analysis of Schweinsberg's model under the more general assumption
(\ref{maincond}) for the interested reader.\\
In contrast, conditional branching process models
\cite{huilletmoehlepop,huilletmoehlecor,huilletmoehleasy} seem to be
harder to analyse and behave quite differently in general. Even for the
subclass of so-called compound Poisson models, only partial results are
available. Theorems 2.2 and 2.3 of \cite{huilletmoehleasy} clarify
that many unbiased compound Poisson models are in the domain of attraction of
the Kingman coalescent, and \cite[Theorem 2.5]{huilletmoehleasy} (subcritical
case) demonstrates that the limiting behaviour of compound Poisson models
can differ substantially from all scenarios arising in Theorem \ref{main}.
To the best of the authors knowledge, the limiting behaviour of the ancestral
structure of unbiased conditional branching process models as $N\to\infty$
under assumptions of the form (\ref{maincond}) has not been fully addressed
in the literature. We leave this analysis for future research.
\subsection{Examples} \label{examples}
\begin{example} (Pareto distribution) \label{pareto}
   Let $X$ be Pareto distributed with parameter $\alpha>0$ having
   tail probabilities $\pr(X>x)=x^{-\alpha}$, $x>1$. Clearly,
   (\ref{maincond}) holds with $\ell\equiv 1$. so Theorem \ref{main}
   is applicable.
   Note that
   $\me(X^p)<\infty$ if and only if $p<\alpha$ and in this case
   $\me(X^p)=\alpha\int_1^\infty x^{p-\alpha-1}\,{\rm d}x=\alpha/(\alpha-p)$.
   In particular $\mu:=\me(X)=\alpha/(\alpha-1)<\infty$ for $\alpha>1$ and
   $\rho:=\me(X^2)=\alpha/(\alpha-2)<\infty$ for $\alpha>2$.
   By Theorem \ref{main}, for $\alpha\ge 2$ the model is in the
   domain of attraction of the Kingman coalescent, for $\alpha\in[1,2)$
   in the domain of attraction of the $\beta(2-\alpha,\alpha)$-coalescent,
   and for $\alpha\in(0,1)$ in the domain of attraction of the discrete-time
   Poisson--Dirichlet coalescent with parameter $\alpha$.

   Note that $\ell^*(x)=\int_1^x 1/t\,{\rm d}t=\log x$,
   $x>1$. In part (iv) of Theorem \ref{main}, we can therefore
   choose $a_N:=N\log N$ and obtain $c_N\sim\ell(a_N)/\ell^*(a_N)=
   1/\ell^*(a_N)\sim 1/\log N$ as $N\to\infty$. Thus, by Theorem
   \ref{main}, the coalescence probability $c_N$ satisfies
   \[
   c_N\ \sim\ \left\{
      \begin{array}{cl}
         \frac{\rho}{\mu^2N}\ =\ \frac{(\alpha-1)^2}{\alpha(\alpha-2)N}
            & \mbox{if $\alpha>2$,}\\
         \frac{2\ell^*(N)}{\mu^2 N}\ =\ \frac{\log N}{2N}
            & \mbox{if $\alpha=2$,}\\
         \frac{\Gamma(2-\alpha)\Gamma(\alpha+1)}{\mu^\alpha N^{\alpha-1}}
            & \mbox{if $\alpha\in(1,2)$,}\\
         \frac{1}{\log N}
            & \mbox{if $\alpha=1$,}\\
         1-\alpha
            & \mbox{if $\alpha\in(0,1)$.}
      \end{array}
   \right.
   \]
   For $\alpha>2$ these results coincide with Proposition 7 of
   \cite{huillet} with $\beta=0$, for $\alpha=2$ with Proposition 9
   of \cite{huillet}, for $\alpha\in (1,2)$ with Lemma 4 and
   Proposition 5 of \cite{huillet} with $\beta=0$, for $\alpha=1$
   with Proposition 6 of \cite{huillet} with $\beta=0$, and for
   $\alpha\in (0,1)$ with Theorem 3 of \cite{huillet} with $\beta=0$.
\end{example}
The Pareto example is easily generalized in various ways by replacing
$\ell\equiv 1$ by some other slowly varying function. For example, choosing
for $\ell$ (a power of) the logarithm leads to the following example.
\begin{example}
   Fix $\alpha\ge 0$ and assume that $X$ has tail behaviour
   $\pr(X>x)\sim x^{-\alpha}\ell(x)$ as $x\to\infty$ with $\ell(x)
   :=c(\log x)^{\beta-1}$, $x>1$, for some constants $c>0$ and $\beta>0$.
   This example includes the Pareto model ($c=\beta=1$). Clearly,
   (\ref{maincond}) holds, since $\ell$ slowly varies at $\infty$.
   By Theorem \ref{main}, for $\alpha\ge 2$ the model is in the domain
   of attraction of the Kingman coalescent, for $\alpha\in[1,2)$ in the
   domain of attraction of the $\beta(2-\alpha,\alpha)$-coalescent, for
   $\alpha\in(0,1)$ in the domain of attraction of the discrete-time
   Poisson--Dirichlet coalescent with parameter $\alpha$, and for $\alpha=0$
   in the domain of attraction of the discrete-time star-shaped coalescent.
   Note that
   \[
   \ell^*(x)
   \ =\ \int_1^x \frac{\ell(t)}{t}\,{\rm d}t
   \ =\ c\int_1^x \frac{(\log t)^{\beta-1}}{t}\,{\rm d}t
   \ =\ \frac{c}{\beta}(\log x)^\beta,\qquad x\to\infty.
   \]
   The asymptotics of the coalescence probability $c_N$ as $N\to\infty$
   can hence be obtained from the formulas provided in Theorem \ref{main}.
   In particular, for $\alpha>1$ the asymptotics of $c_N$ depends on the
   concrete value of $\mu:=\me(X)$. For $\alpha=1$ the asymptotics of
   $c_N$ is obtained as follows. The sequence $(a_N)_{N\in\nz}$, defined
   via $a_1:=1$ and $a_N:=(c/\beta)N(\log N)^\beta$ for $N\in\nz\setminus\{1\}$,
   satisfies $\ell^*(a_N)\sim (c/\beta)(\log a_N)^\beta\sim (c/\beta)(\log N)^\beta=a_N/N$
   as $N\to\infty$. By Theorem \ref{main} (iv), the coalescence probability
   $c_N$ satisfies $c_N\sim \ell(a_N)/\ell^*(a_N)\sim\beta/\log N$ as $N\to\infty$.
\end{example}
For illustration three examples with discrete $X$ are provided.
\begin{example} (Yule--Simon distribution)
   Let $X$ be Yule--Simon distributed \cite{kozubowskipodgorski,yule}
   with parameter $\alpha>0$ having distribution
   $\pr(X=k)=\alpha{\rm B}(\alpha+1,k)
   =\alpha\Gamma(\alpha+1)\Gamma(k)/\Gamma(\alpha+1+k)$, $k\in\nz$,
   where ${\rm B}(.,.)$ and $\Gamma(.)$ denote the beta and the gamma
   function respectively. It is easily checked that
   $\pr(X>k)=\Gamma(\alpha+1)\Gamma(k+1)/\Gamma(k+\alpha+1)$, $k\in\nz_0$.
   In particular, $\pr(X>x)\sim\Gamma(\alpha+1)x^{-\alpha}$ as $x\to\infty$.
   Thus, (\ref{maincond}) holds with $\ell\equiv \Gamma(\alpha+1)$.
   Note that $\me((X)_k)<\infty$
   if and only if $k<\alpha$ and in this case $\me((X)_k)=\alpha k!{\rm B}(\alpha-k,k)$.
   In particular, $\mu=\me(X)=\alpha/(\alpha-1)$ for $\alpha>1$ and
   $\me((X)_2)=2\alpha/((\alpha-1)(\alpha-2))$ for $\alpha>2$, which yields
   $\rho=\me(X^2)=\alpha^2/((\alpha-1)(\alpha-2))$ for $\alpha>2$.
   By Theorem \ref{main}, for $\alpha\ge 2$ the model is in the domain of
   attraction of the Kingman coalescent, for $\alpha\in [1,2)$ in the domain
   of attraction of the $\beta(2-\alpha,\alpha)$-coalescent, and for
   $\alpha\in(0,1)$ in the domain of attraction of the discrete-time
   Poisson--Dirichlet coalescent with parameter $\alpha$. Note that
   $\ell^*(x)=\Gamma(\alpha+1)\int_1^x 1/t\,{\rm d}t=\Gamma(\alpha+1)\log x$,
   $x>1$. In part (iv) of Theorem  \ref{main} we can thus choose
   $a_N:=\Gamma(\alpha+1)N\log N$ and obtain $c_N\sim\ell(a_N)/\ell^*(a_N)
   =1/\log a_N\sim 1/\log N$ as $N\to\infty$. Thus, by Theorem \ref{main}, the
   coalescence probability $c_N$ satisfies
   \[
   c_N\ \sim\ \left\{
      \begin{array}{cl}
         \frac{\rho}{\mu^2N}\ =\ \frac{\alpha-1}{(\alpha-2)N}
            & \mbox{if $\alpha>2$,}\\
         \frac{2\ell^*(N)}{\mu^2N}\ =\ \frac{\log N}{N}
            & \mbox{if $\alpha=2$,}\\
         \frac{\Gamma(2-\alpha)(\Gamma(\alpha+1))^2}{\mu^\alpha N^{\alpha-1}}
            & \mbox{if $\alpha\in(1,2)$,}\\
         \frac{1}{\log N}
            & \mbox{if $\alpha=1$,}\\
         1-\alpha
            & \mbox{if $\alpha\in(0,1)$.}
      \end{array}
   \right.
   \]
   The Yule--Simon model is a discrete analog of the Pareto
   model discussed in Example \ref{pareto}. We refer the reader exemplary to
   Kozubowski and Podg\'orski \cite{kozubowskipodgorski} for some further
   information on Sibuya and Yule--Simon distributions.
\end{example}
\begin{example} (Sibuya distribution)
   Let $X$ be Sibuya distributed with parameter $\alpha\in (0,1)$ having
   probability generating function $f(s)=1-(1-s)^\alpha$, $s\in [0,1]$.
   Note that $f(s)=\sum_{k=1}^\infty (-1)^{k-1}\binom{\alpha}{k}s^k$,
   so $X$ takes the value $k\in\nz$ with probability
   $\pr(X=k)=(-1)^{k-1}\binom{\alpha}{k}=\alpha\Gamma(k-\alpha)/(\Gamma(1-\alpha)k!)$.
   The Laplace transform $\psi$ of $X$ satisfies
   $1-\psi(u)=1-f(e^{-u})=(1-e^{-u})^\alpha\sim u^\alpha$ as $u\to 0$, i.e.
   relation (2.1) of Bingham and Doney \cite{binghamdoney} holds with $n=0$,
   $\beta=\alpha\in (0,1)$ and $L\equiv 1$. By Theorem A of
   \cite{binghamdoney} this relation is equivalent (see Eq.~(2.3b) of
   \cite{binghamdoney}) to $\pr(X>x)\sim (\Gamma(1-\alpha))^{-1}x^{-\alpha}$
   as $x\to\infty$, which shows that (\ref{maincond}) holds with
   $\ell\equiv 1/\Gamma(1-\alpha)$.
   Part (v) of Theorem \ref{main} ensures that the model is in the domain
   of attraction of the Poisson--Dirichlet coalescent with parameter
   $\alpha$ and the coalescence probability $c_N$ satisfies
   $c_N\to 1-\alpha$ as $N\to\infty$.
   The same results are valid when $X$ is $\alpha$-stable,
   $\alpha\in (0,1)$, with Laplace transform $\psi(u):=e^{-u^\alpha}$,
   $u\ge 0$, since in this case the same asymptotics $1-\psi(u)\sim u^\alpha$
   as $u\to 0$ holds. In this sense the Sibuya example is a
   discrete version of the $\alpha$-stable case with $\alpha\in(0,1)$.
\end{example}
\begin{example}
   Let $\alpha\in(1,2)$ and $b\in(0,1/(\alpha-1)]$. Assume that $X$ has
   probability generating function $f(s)=(b+1)s+b((1-s)^\alpha-1)$,
   $s\in [0,1]$. Note that $X$ is discrete taking values in $\nz$ with
   probabilities $p_k:=\pr(X=k)$, $k\in\nz$, given by $p_1=b+1-b\alpha$
   and $p_k=b(-1)^k\binom{\alpha}{k}=b\Gamma(k-\alpha)/(\Gamma(-\alpha)k!)$
   for $k\in\{2,3,\ldots\}$. From $f'(s)=b+1-b\alpha(1-s)^{\alpha-1}$ it
   follows that $\mu:=\me(X)=f'(1)=b+1$. The Laplace transform $\psi$ of
   $X$ satisfies $\psi(u)-1+(b+1)u\sim bu^\alpha$ as $u\to 0$, i.e.
   relation (2.1) of Bingham and Doney \cite{binghamdoney} holds with
   $n=1$, $\beta=\alpha-1\in(0,1)$, and $L\equiv b$. By Theorem A of
   \cite{binghamdoney} this relation is equivalent (see Eq.~(2.3b) of
   \cite{binghamdoney}) to $\pr(X>x)\sim b(-\Gamma(1-\alpha))^{-1}x^{-\alpha}$
   as $x\to\infty$, which shows that (\ref{maincond}) holds with $\ell(x)\equiv
   b/(-\Gamma(1-\alpha))$.
   By Theorem \ref{main} (iii) the model is in the domain of attraction
   of the $\beta(2-\alpha,\alpha)$-coalescent and
   $c_N\sim(\alpha-1)\Gamma(\alpha+1)b/(\mu^\alpha N^{\alpha-1})$ as $N\to\infty$.
\end{example}
We close this section with a concrete example belonging to the boundary
case (vi) ($\alpha=0$).
\begin{example}
   Let $\beta>0$. If $\pr(X>x)=1/(1+\log x)^\beta$, $x\ge 1$, then
   $\pr(X>x)\sim\ell(x)$ as $x\to\infty$ with $\ell(x):=1/(\log x)^\beta$.
   By Theorem \ref{main} (vi), the model is in the domain of attraction
   of the discrete-time star-shaped coalescent and $c_N\to 1$ as $N\to\infty$.
\end{example}
\subsection{Proofs} \label{proofs}
The following auxiliary result (Lemma \ref{lem1}) is a modified version
of Lemma 5 of Schweinsberg \cite{schweinsberg3}, adapted to our model.
The result may be also viewed as a weak version of Cram\'er's large
deviation theorem (see, for example, \cite[Theorem 2.2.3]{dembozeitouni}).
Recall that $\mu:=\me(X)\in(0,\infty]$.
\begin{lemma} \label{lem1}
   For every $a\in(0,\mu)$ there exists $q\in(0,1)$ such
   that $\pr(S_N\le aN)\le q^N$ for all $N\in\nz$.
\end{lemma}
\begin{proof}
   Let $f$ denote the moment generating function of $Y:=X/a$, i.e.
   $f(x):=\me(x^Y)$, $x\in [0,1]$. From $\me(x^{S_N/a})\ge
   \int_{\{S_N\le aN\}} x^{S_N/a}\,{\rm d}\pr\ge x^N\pr(S_N\le aN)$
   it follows that $\pr(S_N\le aN)\le x^{-N}\me(x^{S_N/a})=(x^{-1}f(x))^N$
   for all $x\in (0,1]$. Since $f(1)=1$ and $f'(1)=\me(Y)=\mu/a>1$,
   there exists $x_0\in(0,1)$ such that $f(x_0)<x_0$. The result follows
   with $q:=x_0^{-1}f(x_0)$.\hfill$\Box$
\end{proof}
We now prove part (i) of Theorem \ref{main}.
\begin{proof} (of Theorem \ref{main} (i))
   We first verify that $Nc_N\to\rho/\mu^2$ as $N\to\infty$.
   We have
   \[
   Nc_N
   \ =\ N^2\me(W_1^2)
   \ =\ N^2\int_{(0,\infty)}\me\bigg(\bigg(\frac{x}{x+S_{N-1}}\bigg)^2\bigg)
        \,\pr_X({\rm d}x)=\int_{(0,\infty)}f_N(x)\,\pr_X({\rm d}x),
   \]
   where $f_N(x):=\me((x/(x/N+S_{N-1}/N))^2)$. By the law of large numbers,
   $(x/(x/N+S_{N-1}/N))^2\to(x/\mu)^2$ almost surely and, hence, also in
   distribution as $N\to\infty$. For any $r>0$ the map $x\mapsto x\wedge r$
   is bounded and continuous on $[0,\infty)$. Thus,
   \[
   \liminf_{N\to\infty}f_N(x)\ \ge\
   \liminf_{N\to\infty}\me
   \bigg(\bigg(\frac{x}{\frac{x}{N}+\frac{S_{N-1}}{N}}\bigg)^2\wedge r\bigg)
   \ =\ (x/\mu)^2\wedge r.
   \]
   Letting $r\to\infty$ yields $\liminf_{N\to\infty}f_N(x)\ge (x/\mu)^2$.
   Therefore, by Fatou's lemma, $\liminf_{N\to\infty}Nc_N
   =\liminf_{N\to\infty}\int_{(0,\infty)}f_N(x)\,\pr_X({\rm d}x)
   \ge\int_{(0,\infty)}(x/\mu)^2\,\pr_X({\rm d}x)=\rho/\mu^2$.

   In order to see that $\limsup_{N\to\infty}Nc_N\le \rho/\mu^2$ fix
   $a\in(0,\mu)$. By Lemma \ref{lem1} there exist $q\in (0,1)$ such that
   $\pr(S_N\le aN)\le q^N$ for all $N\in\nz$. Therefore,
   \begin{eqnarray*}
      Nc_N
      & = & N^2\me(W_1^2)
      \ =\ N^2\me(W_1^21_{\{S_N\le aN\}}) + N^2\me((X_1/S_N)^21_{\{S_N>aN\}})\\
      & \le & N^2\pr(S_N\le aN) + N^2\me(((X_1/(aN))^2)
      \ \le\ N^2q^N + \frac{\rho}{a^2}\ \to\ \frac{\rho}{a^2},\qquad N\to\infty.
   \end{eqnarray*}
   Thus, $\limsup_{N\to\infty}Nc_N\le\rho/a^2$. Letting $a\uparrow\mu$
   shows that $\limsup_{N\to\infty}Nc_N\le \rho/\mu^2$ and
   $Nc_N\to\rho/\mu^2$ is established.

   It is well-known (see, \cite[Section 4]{moehletotal}) that any sequence of
   Cannings models with population sizes $N$ is in
   the domain of attraction of the Kingman coalescent if and only if
   $\Phi_1^{(N)}(3)/c_N\to 0$ as $N\to\infty$. Thus, we have to verify
   that $\me(W_1^3)/\me(W_1^2)\to 0$ as $N\to\infty$.
   Since $\me(W_1^2)\ge(\me(W_1))^2=1/N^2$ it suffices to verify that
   $N^2\me(W_1^3)\to 0$ as $N\to\infty$. Fix again $a\in (0,\mu)$ and
   choose $q\in(0,1)$ as above. We have
   \[
   N^2\me(W_1^3)
   \ =\ N^2\me(W_1^31_{\{S_N\le aN\}}) + N^2\me(W_1^31_{\{S_N>aN\}}).
   \]
   Since $N^2\me(W_1^31_{\{S_N\le aN\}})\le N^2\pr(S_N\le aN)\le N^2q^N\to 0$
   as $N\to\infty$ it remains to verify that
   $N^2\me(W_1^31_{\{S_N>aN\}})\to 0$ as $N\to\infty$. We have
   \begin{eqnarray*}
      N^2\me(W_1^31_{\{S_N>aN\}})
      & = & N^2\me(X_1^3S_N^{-3}1_{\{S_N>aN,X_1\le aN\}}) +
            N^2\me(W_1^31_{\{S_n>aN,X_1>aN\}})\\
      & \le & \frac{1}{a^3N}\me(X^31_{\{X\le aN\}})+ N^2\pr(X>aN).
   \end{eqnarray*}
   Clearly, $N^2\pr(X>aN)\le a^{-2}\me(X^21_{\{X>aN\}})\to 0$ as $N\to\infty$,
   since $\rho:=\me(X^2)<\infty$. It hence remains to verify that
   $N^{-1}\me(X^31_{\{X\le aN\}})\to 0$ as $N\to\infty$.
   Let $\varepsilon>0$. Choose $L$
   sufficiently large such that $\me(X^21_{\{X>L\}})\le\varepsilon/(2a)$. Then,
   for all $N\in\nz$ with $N\ge 2\rho L/\varepsilon$,
   \begin{eqnarray*}
      N^{-1}\me(X^31_{\{X\le aN\}})
      & = & N^{-1}\me(X^31_{\{X\le aN,X\le L\}}) + N^{-1}\me(X^31_{\{L<X\le aN\}})\\
      & \le & N^{-1}L\rho  + a\me(X^21_{\{X>L\}})
      \ \le\ \frac{\varepsilon}{2} + \frac{\varepsilon}{2}
      \ =\ \varepsilon,
   \end{eqnarray*}
   which shows that $N^{-1}\me(X^31_{\{X\le aN\}})\to 0$ as
   $N\to\infty$.\hfill$\Box$
\end{proof}
We now prepare the proofs of the parts (ii) and (iii) of Theorem
\ref{main}. We need the following two auxiliary results.
\begin{lemma} \label{lem2}
   If (\ref{maincond}) holds for some
   $\alpha\ge 0$ then for all $p>\alpha$,
   \[
   \me\bigg(\bigg(\frac{X}{X+x}\bigg)^p\bigg)\ \sim\
   \frac{\Gamma(\alpha+1)\Gamma(p-\alpha)}{\Gamma(p)}x^{-\alpha}\ell(x),\qquad x\to\infty,
   \]
   and
   \[
   \me\bigg(\bigg(\frac{X}{X\vee x}\bigg)^p\bigg)\ \sim\ \frac{p}{p-\alpha}
   x^{-\alpha}\ell(x),\qquad x\to\infty.
   \]
\end{lemma}
\begin{proof}
   Let $T$ be a nonnegative random variable and $f:[0,\infty)\to\rz$
   a continuous and piecewise continuously differentiable function such
   that $f(T)$ is integrable. Then,
   \begin{eqnarray}
   \me(f(T))-f(0)
   & = & \int_{[0,\infty)}\big(f(x)-f(0)\big)\,\pr_T({\rm d}x)
   \ = \ \int_{[0,\infty)}\int_{[0,x)}
         f'(t)\,\lambda({\rm d}t)\,\pr_T({\rm d}x)\nonumber\\
   & = & \int_{[0,\infty)} f'(t)
         \int_{(t,\infty)} \pr_T({\rm d}x)\,\lambda({\rm d}t)
   \ = \ \int_0^\infty f'(t)\,\pr(T>t)\,{\rm d}t.\label{meanf}
   \end{eqnarray}
   Let $x>0$. Applying (\ref{meanf}) to $T:=X/x$ and
   $f(t):= (t/(t+1))^p$ shows that
   \[
   \me\bigg(\bigg(\frac{X}{X+x}\bigg)^p\bigg)
   \ =\ \int_0^\infty \frac{pt^{p-1}}{(t+1)^{p+1}}\pr(X>xt)\,{\rm d}t.
   \]
   By Theorem 3 of Karamata \cite{karamata3}, applied to the function
   $\varphi(x):=\pr(X>x)$, which is regularly varying at $\infty$ with
   index $\gamma:=-\alpha$, it follows that, as $x\to\infty$,
   \[
   \me\bigg(\bigg(\frac{X}{X+x}\bigg)^p\bigg)\ \sim\
   \pr(X>x)\int_0^\infty \frac{pt^{p-1}}{(t+1)^{p+1}} t^{-\alpha}\,{\rm d}t
   \ =\ \pr(X>x)\frac{\Gamma(\alpha+1)\Gamma(p-\alpha)}{\Gamma(p)}.
   \]
   The same steps, but applied to $f(t):=(t/(t\vee 1))^p$, show that
   \begin{eqnarray*}
      \me\bigg(\bigg(\frac{X}{X\vee x}\bigg)^p\bigg)
      & = & \int_0^\infty f'(t)\pr(X>xt)\,{\rm d}t
      \ \sim \ \pr(X>x)\int_0^\infty f'(t) t^{-\alpha}\,{\rm d}t\\
      & = & \pr(X>x)\int_0^1 pt^{p-\alpha-1}\,{\rm d}t
      \ = \ \pr(X>x)\frac{p}{p-\alpha}.\hspace{2cm}\Box
   \end{eqnarray*}
\end{proof}
\begin{lemma} \label{fundamental}
   For all $j\in\{1,\ldots,N\}$ and $p_1,\ldots,p_j>0$,
   \begin{equation} \label{fundamental1}
      \me(W_1^{p_1}\cdots W_j^{p_j})\ =\ \frac{1}{\Gamma(p)}
      \int_0^\infty u^{p-1}\me(e^{-uS_{N-j}})
      \prod_{i=1}^j\me(X^{p_i}e^{-uX})\,{\rm d}u,
   \end{equation}
   where $p:=p_1+\cdots+p_j$ and $S_0:=0$.
   Moreover, for any fixed $j\in\nz$ the asymptotics of the latter integral
   as $N\to\infty$ is determined by the values of $u$ close to $0$, i.e.,
   for any fixed $j\in\nz$ and $\delta>0$,
   \begin{equation} \label{fundamental2}
      \me(W_1^{p_1}\cdots W_j^{p_j})\ \sim\ \frac{1}{\Gamma(p)}
      \int_0^\delta u^{p-1}\me(e^{-uS_{N-j}})
      \prod_{i=1}^j\me(X^{p_i}e^{-uX})\,{\rm d}u,\qquad N\to\infty.
   \end{equation}
   In particular, for any fixed $\delta>0$
   \begin{equation} \label{fundamentalspecial}
      \frac{1}{N}\ =\ \me(W_1)
      \ \sim\ \int_0^\delta \me(Xe^{-uX})\me(e^{-uS_{N-1}})\,{\rm d}u,
   \qquad N\to\infty.
   \end{equation}
\end{lemma}
\begin{remark}
   The fundamental relation (\ref{fundamental1}) is well-known from several
   references (see, for example, Cortines \cite[Proposition 4.4]{cortines}
   or Huillet \cite{huillet}).
\end{remark}
\begin{proof}
   Let $j\in\{1,\ldots,N\}$ and $p_1,\ldots,p_j>0$. From
   $S_N^{-p}=(\Gamma(p))^{-1}\int_0^\infty u^{p-1}e^{-uS_{N}}\,{\rm d}u$
   it follows that
   \begin{eqnarray*}
      \me(W_1^{p_1}\cdots W_j^{p_j})
      & = & \me(X_1^{p_1}\cdots X_j^{p_j}S_N^{-p})
      \ = \ \frac{1}{\Gamma(p)}\int_0^\infty u^{p-1}\me(X_1^{p_1}\cdots X_j^{p_j}e^{-uS_N})\,{\rm d}u\\
      & = & \frac{1}{\Gamma(p)}\int_0^\infty u^{p-1}\me(e^{-uS_{N-j}})\prod_{i=1}^j\me(X^{p_i}e^{-uX})\,{\rm d}u,
   \end{eqnarray*}
   which is (\ref{fundamental1}). To check (\ref{fundamental2})
   fix $j\in\nz$ and $\delta>0$ and let $\psi$ denote the Laplace transform
   of $X$. Decompose $\me(W_1^{p_1}\cdots W_j^{p_j})=A_N+B_N$ with
   \[
   A_N\ :=\ \frac{1}{\Gamma(p)}\int_0^\delta u^{p-1}
   \me(e^{-uS_{N-j}})\prod_{i=1}^j\me(X^{p_i}e^{-uX})\,{\rm d}u
   \]
   and
   \[
   B_N\ :=\ \frac{1}{\Gamma(p)}\int_\delta^\infty u^{p-1}\me(e^{-uS_{N-j}})\prod_{i=1}^j\me(X^{p_i}e^{-uX})\,{\rm d}u.
   \]
   The map $u\mapsto\me(e^{-uS_{N-1}})$ is non-increasing on
   $[0,\infty)$. Thus,
   \[
   B_N
   \ \le\ \me(e^{-\delta S_{N-j}})\frac{1}{\Gamma(p)}
          \int_\delta^\infty u^{p-1}\prod_{i=1}^j\me(X^{p_i}e^{-uX})\,{\rm d}u
   \ = \ c_1(\psi(\delta))^{N-j}
   \]
   and
   \[
   A_N
   \ \ge\ \frac{1}{\Gamma(p)}\int_0^{\delta/2}
          u^{p-1}\me(e^{-uS_{N-j}})\prod_{i=1}^j\me(X^{p_i}e^{-uX})\,{\rm d}u
   \ \ge \ c_2(\psi(\delta/2))^{N-j}
   \]
   with constants
   \[
   c_1\ :=\ c_1(p_1,\ldots,p_j,\delta)
   \ :=\ \frac{1}{\Gamma(p)}\int_\delta^\infty u^{p-1}\prod_{i=1}^j
   \me(X^{p_i}e^{-uX})\,{\rm d}u
   \]
   and
   \[
   c_2\ :=\ c_2(p_1,\ldots,p_j,\delta)
   \ :=\ \frac{1}{\Gamma(p)}\int_0^{\delta/2}u^{p-1}
   \prod_{i=1}^j\me(X^{p_i}e^{-uX})\,{\rm d}u.
   \]
   Note that $0<c_1,c_2<\infty$ and that $c_1$ and $c_2$ do not depend on $N$.
   Thus,
   \[
   1\ \le\ \frac{\me(W_1^{p_1}\cdots W_j^{p_j})}{A_N}
   \ =\ 1 + \frac{B_N}{A_N}
   \ \le\ 1 + \frac{c_1(\psi(\delta))^{N-j}}{c_2(\psi(\delta/2))^{N-j}}
   \ =\ 1 + \frac{c_1}{c_2}\bigg(\frac{\psi(\delta)}{\psi(\delta/2)}\bigg)^{N-j}
   \ \to\ 1
   \]
   as $N\to\infty$, since $\psi(\delta/2)>\psi(\delta)$. Eq.~(\ref{fundamentalspecial}) follows from by choosing $j:=1$ and $p_1:=1$ in (\ref{fundamental2}).\hfill$\Box$
\end{proof}
We now turn to the proofs of the parts (ii) and (iii) of Theorem \ref{main}.
We first consider part (iii) ($1<\alpha<2$). The boundary case $\alpha=2$
(part (ii) of Theorem \ref{main}) will be studied afterwards.
\begin{proof} (of Theorem \ref{main} (iii))
   The idea of the proof is to apply the general convergence result
   \cite[Theorem 2.1]{moehlesagitov}. Having (\ref{transprob}) in mind the
   main task is to derive the asymptotics of the moments of $W_1$ or, more
   generally, the asymptotics of the joint moments of the random
   variables $W_1,\ldots,W_j$ as $N\to\infty$. The following
   proof is based on Schweinsberg's \cite{schweinsberg3} method.
%   We provide two different proofs.
%   The first proof is based on Schweinsberg's \cite{schweinsberg3} method.
%   The second proof is based on Laplace transforms and exploits methods from
%   Huillet \cite{huillet}.
%
%   \vspace{2mm}
%
%   {\bf First proof.} (via Schweinsberg's \cite{schweinsberg3} method)
   We first verify that
   \begin{equation} \label{momlimit}
   \lim_{N\to\infty}\frac{(\mu N)^\alpha}{\ell(N)}\me(W_1^k)\ =\
   \alpha{\rm B}(k-\alpha,\alpha),\qquad k\in\nz\setminus\{1\}.
   \end{equation}
   For all $\lambda>\mu:=\me(X)$, by the law of large numbers,
   $\pr(S_{N-1}\le\lambda N)\to 1$ as $N\to\infty$. Thus,
   \begin{eqnarray*}
      \me(W_1^k)
      & \ge & \me(W_1^k1_{\{X_2+\cdots+X_N\le\lambda N\}})
      \ \ge \ \me\bigg(\bigg(\frac{X_1}{X_1+\lambda N}\bigg)^k\bigg)
              \pr(X_2+\cdots+X_N\le\lambda N)\\
      & \sim & \me\bigg(\bigg(\frac{X}{X+\lambda N}\bigg)^k\bigg)
      \ \sim\ \alpha{\rm B}(k-\alpha,\alpha)\frac{\ell(N)}{(\lambda N)^\alpha},
      \qquad N\to\infty,
   \end{eqnarray*}
   where the last asymptotics holds by
   Lemma \ref{lem2} and since $\ell$ is slowly varying at $\infty$. Multiplication with $N^\alpha/\ell(N)$
   and taking $\liminf$ shows that
   $\liminf_{N\to\infty}N^\alpha/\ell(N)\me(W_1^k)\ge \alpha {\rm B}(k-\alpha,\alpha)/\lambda^\alpha$. Letting $\lambda\downarrow\mu$
   it follows that
   $\liminf_{N\to\infty}N^\alpha/\ell(N)\me(W_1^k)\ge \alpha{\rm B}(k-\alpha,\alpha)/\mu^\alpha$.

   To handle the $\limsup$ fix $a\in(0,\mu)$ and decompose
   \[
   \me(W_1^k)\ =\ \me(W_1^k1_{\{X_2+\cdots+X_N\le aN\}}) +
   \me(W_1^k1_{\{X_2+\cdots+X_N>aN\}}).
   \]
   From Lemma
   \ref{lem1} it follows that there exists $N_0\in\nz$ and $q\in(0,1)$ such that $\pr(S_{N-1}\le aN)\le q^N$
   for all $N>N_0$. Thus,
   $\me(W_1^k1_{\{X_2+\cdots+X_N\le aN\}})\le\pr(X_2+\cdots+X_N\le aN)=\pr(S_{N-1}\le aN)\le q^N$ for all
   $N\in\nz$ with $N>N_0$. It hence suffices to verify that
   \begin{equation} \label{mylocal}
      \limsup_{N\to\infty}\frac{(\mu N)^\alpha}{\ell(N)}\me(W_1^k1_{\{X_2+\cdots+X_N>aN\}})\ =\ \alpha{\rm B}(k-\alpha,\alpha).
   \end{equation}
   In order to see this let $\lambda\in (a,\mu)$ and decompose
   \begin{eqnarray*}
      &   & \hspace{-15mm}\me(W_1^k1_{\{X_2+\cdots+X_N>aN\}})
      \ = \ \me(W_1^k1_{\{aN<X_2+\cdots+X_N\le\lambda N\}})
            + \me(W_1^k1_{\{X_2+\cdots+X_N>\lambda N\}})\\
      & \le & \me\bigg(\bigg(\frac{X_1}{X_1+aN}\bigg)^k\bigg)\pr(S_{N-1}\le\lambda N) + \me\bigg(\bigg(\frac{X_1}{X_1+\lambda N}\bigg)^k\bigg)\pr(S_{N-1}>\lambda N).
   \end{eqnarray*}
   The two expectations on the right hand side are both $O(\ell(N)/N^\alpha)$ by
   Lemma \ref{lem2}. Moreover, $\pr(S_{N-1}\le\lambda N)\to 0$ and $\pr(S_{N-1}>\lambda N)\to 1$ as $N\to\infty$. Therefore, only the last term contributes to the $\limsup$ and we obtain
   \begin{eqnarray*}
      \limsup_{N\to\infty}\frac{N^\alpha}{\ell(N)}\me(W_1^k1_{\{X_2+\cdots+X_N>aN\}})
      & \le & \limsup_{N\to\infty} \frac{N^\alpha}{\ell(N)}
      \me\bigg(\bigg(\frac{X_1}{X_1+\lambda N}\bigg)^k\bigg)\pr(S_{N-1}>\lambda N)\\
      & \sim & \frac{N^\alpha}{\ell(N)}\alpha {\rm B}(k-\alpha,\alpha)
      \frac{\ell(\lambda N)}{(\lambda N)^\alpha}
      \ =\ \alpha{\rm B}(k-\alpha,\alpha) /\lambda^\alpha.
   \end{eqnarray*}
   Letting $\lambda\uparrow\mu$ shows that (\ref{mylocal}) holds.
   Thus, (\ref{momlimit}) is established.

   Choosing $k=2$ in (\ref{momlimit}) yields the asymptotic formula for
   the coalescence probability $c_N=N\me(W_1^2)$ stated in Theorem
   \ref{main} (iii). In particular, $c_N=O(\ell(N)/N^{\alpha-1})$.
   In summary we conclude that
   \[
   \frac{\Phi_1^{(N)}(k)}{c_N}
   \ =\ \frac{\me(W_1^k)}{\me(W_1^2)}
   \ \to\ \frac{\Gamma(k-\alpha)}{\Gamma(k)\Gamma(2-\alpha)}
   \ =\ \int_{(0,1)} x^{k-2}\Lambda({\rm d}x),\qquad N\to\infty,
   \]
   where $\Lambda:=\beta(2-\alpha,\alpha)$ denotes the beta distribution
   with parameters $2-\alpha$ and $\alpha$.

   Moreover,
   \begin{eqnarray*}
      \me(W_1^2W_2^21_{\{S_N>aN\}})
      & \le & \me\bigg(\frac{X_1^2X_2^2}{(X_1\vee aN)^2(X_2\vee aN)^2}\bigg)\\
      & = & \bigg(\me\bigg(\frac{X^2}{(X\vee aN)^2}\bigg)\bigg)^2
      \ \sim\ \bigg(\frac{2}{2-\alpha}
           \frac{\ell(aN)}{(aN)^\alpha}\bigg)^2
      \ = \ O\bigg(\frac{(\ell(N))^2}{N^{2\alpha}}\bigg).
   \end{eqnarray*}
   Since $c_N\ge K\ell(N)/N^{\alpha-1}$ for some constant $K>0$ it follows that
   $\Phi_2^{(N)}(2,2)/c_N=O(\ell(N)/N^{\alpha-1})=O(c_N)\to 0$ as $N\to\infty$.
   Thus, for all $j,k_1,\ldots,k_j\in\nz\setminus\{1\}$,
   $\Phi_j^{(N)}(k_1,\ldots,k_j)/c_N\le
   \Phi_2^{(N)}(2,2)/c_N\to 0$ as $N\to\infty$.
   By \cite[Theorem 2.1]{moehlesagitov}, the model is in the
   domain of attraction of the $\beta(2-\alpha,\alpha)$-coalescent.\hfill$\Box$
\end{proof}
We now turn to the boundary case $\alpha=2$, so we prove part (ii) of Theorem
\ref{main}.
\begin{proof} (of Theorem \ref{main} (ii))
   For all $x>0$,
   \begin{eqnarray*}
      \me(X^21_{\{X\le x\}})
      & = & \int_0^\infty \pr(X^21_{\{X\le x\}}>y)\,{\rm d}y
      \ = \ \int_0^\infty 2t\pr(X^21_{\{X\le x\}}>t^2)\,{\rm d}t\\
      & = & \int_0^x 2t\pr(t<X\le x)\,{\rm d}t
      \ = \ \int_0^x 2t\big(\pr(X>t) - \pr(X>x)\big)\,{\rm d}t\\
      & = & \int_0^x 2t\pr(X>t)\,{\rm d}t - x^2\pr(X>x).
   \end{eqnarray*}
   Since $\int_0^xt\pr(X>t)\,{\rm d}t\sim\int_1^x\ell(t)/t\,{\rm d}t=\ell^*(x)$,
   $x^2\pr(X>x)\sim\ell(x)$, and $\ell(x)/\ell^*(x)\to 0$ as $x\to\infty$,
   it follows that $\me(X^21_{\{X\le x\}})\sim 2\ell^*(x)$ as $x\to\infty$.
   Thus, relation (2.3c) of Bingham and Doney \cite{binghamdoney} holds with
   $n=1$ and $L:=\ell^*$. This relation is equivalent (see (2.4) in Theorem A of
   \cite{binghamdoney})
   to $\psi''(u)\sim 2\ell^*(1/u)$ as $u\to 0$.

   Recall that $c_N=N\me(W_1^2)$.
   We now verify the asymptotic relation $c_N\sim 2\mu^{-2}\ell^*(N)/N$ as
   $N\to\infty$ or, equivalently, that
   \begin{equation} \label{equivalently}
      \lim_{N\to\infty}\frac{N^2}{\ell^*(N)}\me(W_1^2)\ =\ \frac{2}{\mu^2}.
   \end{equation}
   We have
   \[
   \me(W_1^2)\ =\ \int_0^\infty u\psi''(u)\me(e^{-uS_{N-1}})\,{\rm d}u
   \ =\ \frac{1}{N^2}\int_0^\infty t\psi''(t/N)\me(e^{-tS_{N-1}/N})\,{\rm d}t.
   \]
   Multiplication with $N^2/\ell^*(N)$ and Fatou's lemma yields
   \[
   \liminf_{N\to\infty}\frac{N^2}{\ell^*(N)}\me(W_1^2)
   \ \ge\ \int_0^\infty t\liminf_{N\to\infty}
         \frac{\psi''(t/N)}{\ell^*(N)}\me(e^{-tS_{N-1}/N})\,{\rm d}t
   \ =\ \int_0^\infty 2te^{-\mu t}\,{\rm d}t
   \ =\ \frac{2}{\mu^2},
   \]
   since $\psi''(t/N)\sim 2\ell^*(N/t)\sim 2\ell^*(N)$ and
   $\me(e^{-tS_{N-1}/N})\to e^{-\mu t}$ as $N\to\infty$. To see that
   $\limsup_{N\to\infty}\frac{N^2}{\ell^*(N)}\me(W_1^2)\le 2/\mu^2$
   fix $a\in (0,\mu)$. By Lemma \ref{lem1} there exists $N_0\in\nz$ and
   $q\in (0,1)$ such that $\pr(S_{N-1}\le aN)\le q^N$ for all $N\in\nz$ with
   $N>N_0$. Noting that $\me(W_1^21_{ \{X_2+\cdots+X_N\le aN\}})\le
   \pr(S_{N-1}\le aN)\le q^N$ it suffices to verify that
   \begin{equation} \label{suffices}
      \limsup_{N\to\infty}\frac{N^2}{\ell^*(N)}\me(W_1^21_{\{X_2+\cdots+X_N>aN\}})
      \ \le\ \frac{2}{\mu^2}.
   \end{equation}
   In order to see this let $\lambda\in (a,\mu)$ and decompose
   $\me(W_1^21_{\{X_2+\cdots+X_N>aN\}})=\me(W_1^21_{A_N})+\me(W_1^21_{B_N})$,
   where $A_N:=\{aN<X_2+\cdots+X_N\le\lambda N\}$ and $B_N:=\{X_2+\cdots+X_N>\lambda N\}$.
   We have
   \begin{eqnarray*}
      \frac{N^2}{\ell^*(N)}\me(W_1^21_{A_N})
      & = & \frac{N^2}{\ell^*(N)}\int_0^\infty
            u\psi''(u)\me(e^{-uS_{N-1}}1_{\{aN<S_{N-1}\le\lambda N\}})
            \,{\rm d}u\\
      & \le & \pr(S_{N-1}\le\lambda N)\frac{N^2}{\ell^*(N)}\int_0^\infty u\psi''(u) e^{-uaN}\,{\rm d}u\\
      & = & \pr(S_{N-1}\le\lambda N)\frac{1}{\ell^*(N)}\int_0^\infty t\psi''(t/N)e^{-at}\,{\rm d}t\\
      & \sim & \pr(S_{N-1}\le\lambda N)\frac{1}{\ell^*(N)}\psi''(1/N)\int_0^\infty te^{-at}\,{\rm d}t\\
      & \sim & \pr(S_{N-1}\le\lambda N)\frac{2}{a^2}
      \ \to\ 0,\qquad N\to\infty,
   \end{eqnarray*}
   where the second last asymptotics holds by Theorem 3 of Karamata
   \cite{karamata3}, applied with $f(t):=te^{-at}$ and $\varphi:=\psi''$,
   which is slowly varying at $0$. For the second part we obtain
   \begin{eqnarray*}
      \frac{N^2}{\ell^*(N)}\me(W_1^21_{B_N})
      & = & \frac{N^2}{\ell^*(N)}\int_0^\infty
            u\psi''(u)\me(e^{-uS_{N-1}}1_{\{S_{N-1}>\lambda N\}})\,{\rm d}u\\
      & \le & \frac{N^2}{\ell^*(N)}\int_0^\infty u\psi''(u)e^{-u\lambda N}\,{\rm d}u
      \ =\ \frac{1}{\ell^*(N)}\int_0^\infty t\psi''(t/N) e^{-\lambda t}\,{\rm d}t\\
      & \sim & \frac{1}{\ell^*(N)}\psi''(1/N)\int_0^\infty te^{-\lambda t}\,{\rm d}t
      \ \sim\ \frac{2}{\lambda^2},
   \end{eqnarray*}
   where the second last asymptotics holds again by Theorem 3 of Karamata
   \cite{karamata3}, now applied with $f(t):=te^{-\lambda t}$ and
   $\varphi:=\psi''$. Therefore,
   \begin{eqnarray*}
      &   & \hspace{-15mm}\limsup_{N\to\infty}
            \frac{N^2}{\ell^*(N)}\me(W_1^21_{\{X_2+\cdots+X_N>aN\}})\\
      & \le & \limsup_{N\to\infty}\frac{N^2}{\ell^*(N)} \me(W_1^21_{A_N}) +
          \limsup_{N\to\infty}\frac{N^2}{\ell^*(N)}\me(W_1^21_{B_N})
      \ \le \ 0 + \frac{2}{\lambda^2}\ =\ \frac{2}{\lambda^2}.
   \end{eqnarray*}
   Letting $\lambda\uparrow\mu$ shows that (\ref{suffices}) holds. Thus,
   (\ref{equivalently}) is established. The rest of the proof now works
   as follows. By the monotone density theorem (Lemma \ref{mdt}), applied
   with $\rho=0$,
   \[
   \frac{-u\psi'''(u)}{\psi''(u)}
   \ \sim\ \frac{-u\psi'''(u)}{2\ell^*(1/u)}\ \to\ 0,\qquad u\to 0.
   \]
   Thus, for every $\varepsilon>0$ there exists $\delta=\delta(\varepsilon)>0$
   such that $-u\psi'''(u)\le\varepsilon \psi''(u)$ for all $u\in (0,\delta)$.
   Therefore, together with Lemma \ref{fundamental}, as $N\to\infty$,
   \[
      \me(W_1^3)
      \ \sim\ \frac{1}{2}\int_0^\delta
            u^2(-\psi'''(u))(\psi(u))^{N-1}\,{\rm d}u
      \ \le\ \frac{\varepsilon}{2}\int_0^\delta
            u\psi''(u)(\psi(u))^{N-1}\,{\rm d}u
      \ \sim \ \frac{\varepsilon}{2}\me(W_1^2).
   \]
   Thus, $\limsup_{N\to\infty}\me(W_1^3)/\me(W_1^2)\le\varepsilon/2$.
   Since $\varepsilon$ can be chosen arbitrarily small, it follows that
   $\lim_{N\to\infty}\Phi_1^{(N)}(3)/c_N=\lim_{N\to\infty}\me(W_1^3)/\me(W_1^2)=0$,
   which is equivalent (see, for example, \cite[Section 4]{moehletotal}) to
   the property that the model is in the domain of attraction of the Kingman
   coalescent.\hfill$\Box$
\end{proof}
We now turn to the proofs of the three remaining parts (iv) - (vi) of
Theorem \ref{main}. We first consider the case $0<\alpha<1$ corresponding
to part (v) of Theorem \ref{main}. The boundary cases (iv) ($\alpha=1$)
and (vi) ($\alpha=0$) will be considered afterwards. Assume that
$0<\alpha<1$. Then (\ref{maincond}) is exactly Eq.~(2.3b) of Bingham and
Doney \cite{binghamdoney} with $n=0$, $\beta=\alpha\in (0,1)$ and
$L(x):=\Gamma(1-\alpha)\ell(x)$.
By \cite[Theorem A]{binghamdoney}, (\ref{maincond}) is hence equivalent
(see \cite[Eq.~(2.1)]{binghamdoney}) to
$1-\psi(u)\sim u^\alpha L(1/u)=\Gamma(1-\alpha)u^\alpha\ell(1/u)$ as $u\to 0$.
\begin{proof} (of Theorem \ref{main} (v))
   For $k\in\nz_0$ and $x>0$ define $h_k(x):=x^k\pr(X>x)$.
   By (\ref{maincond}), $h_k(x)\sim x^{k-\alpha}\ell(x)$ as $x\to\infty$.
   Karamata's Tauberian theorem \cite[Theorem 1.7.6]{binghamgoldieteugels},
   applied with $U:=h_k$, $\rho:=k-\alpha$ and $c:=\Gamma(\rho+1)$,
   yields for all $k\in\nz_0$ that $\widehat{h}_k(u)
   :=u\int_0^\infty e^{-ux}x^k\pr(X>x)\,{\rm d}x\sim
   \Gamma(k-\alpha+1)u^{\alpha-k}\ell(1/u)$ as $u\to 0$. Thus, by
   (\ref{meanf}), for all $k\in\nz$,
   \begin{eqnarray}
      \varphi_k(u)
      & := & \me(X^ke^{-uX})
      \ = \ \int_0^\infty
            \frac{{\rm d}}{{\rm d}x}(x^ke^{-ux})\pr(X>x)\,{\rm d}x\nonumber\\
      & = & \int_0^\infty (kx^{k-1}e^{-ux}-ux^ke^{-ux})\pr(X>x)\,{\rm d}x
      \ = \ \frac{k}{u}\widehat{h}_{k-1}(u) - \widehat{h}_k(u)\nonumber\\
      & \sim & \frac{k}{u}\Gamma(k-\alpha)u^{\alpha-(k-1)}\ell(1/u)
            - \Gamma(k-\alpha+1)u^{\alpha-k}\ell(1/u)\nonumber\\
      & = & \alpha\Gamma(k-\alpha)u^{\alpha-k}\ell(1/u),\qquad u\to 0.
      \label{asy2}
   \end{eqnarray}
   We now turn to the joint moments of $W_1,\ldots,W_j$.
   Let $a_1,a_2,\ldots$ be positive real numbers satisfying
   $L(a_N)\sim a_N^\alpha/N$ as $N\to\infty$.
   Moreover, fix some $\delta\in(0,\infty)$. The exact value
   of $\delta$ is irrelevant but it is important that $\delta$ is finite.
   Let $j,k_1,\ldots,k_j\in\nz$. Define $k:=k_1+\cdots+k_j$. By Lemma \ref{fundamental}, as $N\to\infty$,
   \begin{eqnarray*}
   \Phi_j^{(N)}(k_1,\ldots,k_j)
   & = & (N)_j\me(W_1^{k_1}\cdots W_j^{k_j})
   \ \sim \ \frac{N^j}{\Gamma(k)}\int_0^\delta u^{k-1}
         \me(e^{-uS_{N-j}})\prod_{i=1}^j\varphi_{k_i}(u) \,{\rm d}u\\
   & = & \frac{N^j}{\Gamma(k)a_N^k}
         \int_0^{\delta a_N} t^{k-1}\me(e^{-tS_{N-j}/a_N})
         \prod_{i=1}^j \varphi_{k_i}(t/a_N)\,{\rm d}t.
   \end{eqnarray*}
   Corollary \ref{karamatacorollary}, an Abelian result \'a la Karamata provided in the appendix for convenience, applied with $x_N:=1/a_N$, $f_N(t):=t^{k-1}\me(e^{-tS_{N-j}/a_N})1_{(0,\delta a_N)}(t)$
   and $\varphi:=\prod_{i=1}^j\varphi_{k_i}$, which is regularly varying at
   $0$ with index $\sum_{i=1}^j(\alpha-k_i)=j\alpha-k$, yields
   \begin{equation} \label{kara}
      \Phi_j^{(N)}(k_1\ldots,k_j)
      \ \sim\ \frac{N^j\prod_{i=1}^j\varphi_{k_i}(1/a_N)}{\Gamma(k)a_N^k}
      \int_0^{\delta a_N} t^{j\alpha-1}\me(e^{-tS_{N-j}/a_N})\,{\rm d}t,
      \qquad N\to\infty.
   \end{equation}
   In the following the asymptotic relation (\ref{kara}) is used
   to verify by induction on $j\in\nz$ that, for all
   $k_1,\ldots,k_j\in\nz$,
   \begin{equation} \label{phiconv}
      \lim_{N\to\infty}\Phi_j^{(N)}(k_1\ldots,k_j)
      \ =\ \alpha^{j-1}\frac{\Gamma(j)}{\Gamma(k)}\prod_{i=1}^j \frac{\Gamma(k_i-\alpha)}{\Gamma(1-\alpha)}.
   \end{equation}
   Since $\Phi_1^{(N)}(1)=N\me(W_1)=1$, the choice $j=k_1=1$ in (\ref{kara}) yields
   \begin{equation} \label{kara1}
      \int_0^{\delta a_N} t^{\alpha-1}\me(e^{-tS_{N-1}/a_N})\,{\rm d}t
      \ \sim\ \frac{a_N}{N\varphi_1(1/a_N)}\ \sim\ \frac{1}{\alpha},
      \qquad N\to\infty,
   \end{equation}
   where the last asymptotics holds, since $\varphi_1(1/a_N)\sim
   \alpha\Gamma(1-\alpha)a_N^{1-\alpha}\ell(a_N)$ and $a_N^\alpha/N\sim L(a_N)
   =\Gamma(1-\alpha)\ell(N)$. Note that in (\ref{kara1}) it is important that
   $\delta<\infty$ because otherwise the integral on the left hand side of
   (\ref{kara1}) could take the value $\infty$.
   For $j=1$ and $k_1=k\in\nz$, (\ref{kara}) thus reduces to
   \[
   \Phi_1^{(N)}(k)
   \ \sim\ \frac{N\varphi_k(1/a_N)}{\Gamma(k)a_N^k}\frac{1}{\alpha}
   \ \sim\ \frac{\Gamma(k-\alpha)}{\Gamma(k)\Gamma(1-\alpha)},\qquad N\to\infty,
   \]
   which shows that (\ref{phiconv}) holds for $j=1$.
   In particular, $c_N=\Phi_1^{(N)}(2)\to 1-\alpha>0$ as $N\to\infty$.
   The induction step from $j-1$ to $j$ ($\ge 2$) works as follows.
   By the consistency relation (\ref{consistent}) and the induction
   hypothesis,
   \[
   \Phi_j^{(N)}(1,\ldots,1)
   \ =\ \Phi_{j-1}^{(N)}(1,\ldots,1)-(j-1)\Phi_{j-1}^{(N)}(2,1,\ldots,1)
   \ \to\
   \alpha^{j-2}-\alpha^{j-2}(1-\alpha)=\alpha^{j-1}.
   \]
   Thus, (\ref{phiconv}) holds for $k_1=\cdots=k_j=1$ and the choice
   $k_1=\cdots=k_j=1$ in (\ref{kara}) yields
   \[
   \int_0^{\delta a_N} t^{j\alpha-1}\me(e^{-tS_{N-j}/a_N})\,{\rm d}t\ \sim\
   \ \frac{\Gamma(j)a_N^j}{N^j(\varphi_1(1/a_N))^j}\alpha^{j-1}
   \ \sim\ \frac{\Gamma(j)}{\alpha},
   \qquad N\to\infty.
   \]
   Therefore, (\ref{kara}) reduces to
   \begin{eqnarray*}
   \Phi_j^{(N)}(k_1,\ldots,k_j)
   & \sim & \frac{N^j\prod_{i=1}^j\varphi_{k_i}(1/a_N)}{\Gamma(k)a_N^k}
            \frac{\Gamma(j)}{\alpha}\\
   & \sim & \frac{N^j\Gamma(j)\prod_{i=1}^j
            (\alpha\Gamma(k_i-\alpha)a_N^{k_i-\alpha}\ell(a_N))}
            {\alpha\Gamma(k)a_N^k}\\
   & = & \alpha^{j-1}\frac{\Gamma(j)}{\Gamma(k)}
         \Big(\frac{N\Gamma(1-\alpha)\ell(a_N)}{a_N^\alpha}\Big)^j
         \prod_{i=1}^j \frac{\Gamma(k_i-\alpha)}{\Gamma(1-\alpha)}\\
   & \to & \alpha^{j-1}\frac{\Gamma(j)}{\Gamma(k)}\prod_{i=1}^j
         \frac{\Gamma(k_i-\alpha)}{\Gamma(1-\alpha)}
         \ =:\ \phi_j(k_1,\ldots,k_j),
   \end{eqnarray*}
   since $N\Gamma(1-\alpha)\ell(a_N)=NL(a_N)\sim a_N^\alpha$ as $N\to\infty$.
   The induction is complete.\\
   In summary, $\Phi_j^{(N)}(k_1,\ldots,k_j)\to\phi_j(k_1,\ldots,k_j)$
   as $N\to\infty$
   for all $j,k_1,\ldots,k_j\in\nz$. The quantities $\phi_j(k_1,\ldots,k_j)$ are (see, \cite[Eq.~(16)]{moehleproper}
   for the analog formula for the rates of the continuous-time
   Poisson--Dirichlet coalescent) the transition probabilities of the
   discrete-time two-parameter Poisson--Dirichlet coalescent with parameters
   $\alpha$ and $0$. The convergence result (v) of Theorem \ref{main}
   therefore follows from \cite[Theorem 2.1]{moehlesagitov}.\hfill$\Box$
\end{proof}
Let us now turn to the (boundary) case $\alpha=1$, so we now assume that
$\pr(X>x)\sim x^{-1}\ell(x)$ as $x\to\infty$ for some function $\ell$
slowly varying at $\infty$.
\begin{proof} (of Theorem \ref{main} (iv))
   The proof has much in common with that of part (v). The details are
   however slightly different. For all $x>0$,
   \begin{eqnarray*}
      \me(X1_{\{X\le x\}})
      & = & \int_0^\infty \pr(X1_{\{X\le x\}}>t)\,{\rm d}t
      \ = \ \int_0^x \pr(t<X\le x)\,{\rm d}t\\
      & = & \int_0^x \big(\pr(X>t)-\pr(X>x)\big)\,{\rm d}t
      \ = \ \int_0^x \pr(X>t)\,{\rm d}t - x\pr(X>x).
   \end{eqnarray*}
   Since $\int_0^x \pr(X>t)\,{\rm d}t\sim\int_1^x \ell(t)/t\,{\rm d}t=
   \ell^*(x)$, $x\pr(X>x)\sim\ell(x)$ and $\ell(x)/\ell^*(x)\to 0$ as
   $x\to\infty$ it follows that $\me(X1_{\{X\le x\}})\sim\ell^*(x)$ as
   $x\to\infty$. Recall that $\ell^*$ is
   slowly varying at $\infty$.
   Thus, Eq.~(2.3c) of Bingham and Doney \cite{binghamdoney}
   holds with $n=0$ and $\alpha=\beta=1$ and $L:=\ell^*$,
   which is equivalent (see
   \cite[Theorem A, Eq.~(2.1)]{binghamdoney}) to
   \[
   1-\psi(u)\ \sim\ u\ell^*(1/u),\qquad u\to 0
   \] and as well
   (see \cite[Theorem A, Eq.~(2.4)]{binghamdoney}) equivalent to
   \[
   \varphi_1(u)\ :=\ \me(Xe^{-uX})\ =\ -\psi'(u)\ \sim\ \ell^*(1/u),
   \qquad u\to 0.
   \]
   For $k\in\nz\setminus\{1\}$, the asymptotic relation
   \begin{equation} \label{tilde}
      \varphi_k(u)\ :=\ \me(X^ke^{-uX})
      \ \sim\ \Gamma(k-1)u^{1-k}\ell(1/u),\qquad u\to 0,
   \end{equation}
   is verified exactly as in the proof of part (v) of Theorem \ref{main}.
   In particular, $\varphi_k$ is regularly varying at $0$ with index
   $1-k$, $k\in\nz_0$.\\
   We now turn to the joint moments of $W_1,\ldots,W_j$. Let
   $a_1,a_2,\ldots$ be positive real numbers satisfying
   $\ell^*(a_N)\sim a_N/N$ as $N\to\infty$. As in the proof of part (v) of
   Theorem \ref{main},
   fix some $\delta\in(0,\infty)$. Again, the exact value of $\delta$ is irrelevant but it
   is important that $\delta$ is finite.
   Let $j,k_1,\ldots,k_j\in\nz$. Define $k:=k_1+\cdots+k_j$. By Lemma \ref{fundamental}, as $N\to\infty$,
   \begin{eqnarray*}
      \me(W_1^{k_1}\cdots W_j^{k_j})
      & \sim & \frac{1}{\Gamma(k)}\int_0^\delta
          u^{k-1}\me(e^{-uS_{N-j}})\prod_{i=1}^j\varphi_{k_i}(u)\,{\rm d}u\\
      & = & \frac{1}{\Gamma(k)a_N^k}\int_0^{\delta a_N}
          t^{k-1}
          \me(e^{-tS_{N-j}/a_N})\prod_{i=1}^j\varphi_{k_i}(t/a_N)\,{\rm d}t.
   \end{eqnarray*}
   Corollary \ref{karamatacorollary}, applied with $x_N:=1/a_N$,
   $f_N(t):=t^{k-1}\me(e^{-tS_{N-j}/a_N})1_{(0,\delta a_N)}(t)$ and
   $\varphi:=\prod_{i=1}^j\varphi_{k_i}$, which is regularly varying
   at $0$ with index $\sum_{i=1}^j(1-k_i)=j-k$, shows that
   \begin{equation} \label{first}
      \me(W_1^{k_1}\cdots W_j^{k_j})
      \ \sim\ \frac{\prod_{i=1}^j\varphi_{k_i}(1/a_N)}{\Gamma(k)a_N^k}
      \int_0^{\delta a_N} t^{j-1}\me(e^{-tS_{N-j}/a_N})\,{\rm d}t,
      \qquad N\to\infty.
   \end{equation}
   Since $\me(W_1)=1/N$, the asymptotic relation (\ref{first}) turns for $j=k_1=1$ into
   $1/N\sim a_N^{-1}\varphi_1(1/a_N)
   \int_0^{\delta a_N}\me(e^{-tS_{N-1}/a_N})\,{\rm d}t\sim a_N^{-1}\ell^*(a_N)\int_0^{\delta a_N}
   \me(e^{-tS_{N-1}/a_N})\,{\rm d}t$, or, equivalently,
   \[
   \int_0^{\delta a_N} \me(e^{-tS_{N-1}/a_N})\,{\rm d}t\ \sim\ \frac{a_N}{N\ell^*(a_N)}
   \ \sim\ 1, \qquad N\to\infty.
   \]
   Therefore, for $j=1$ and $k=k_1\in\nz\setminus\{1\}$, (\ref{first}) reduces to
   \[
   \me(W_1^k)\ \sim\ \frac{\varphi_k(1/a_N)}{\Gamma(k)a_N^k}\
   \ \sim\  % \frac{\Gamma(k-1)\ell(a_N)}{\Gamma(k)a_N}\frac{1}{a_N}
      \frac{\ell(a_N)}{(k-1)a_N},\qquad N\to\infty,
   \]
   since $\varphi_k(1/a_N)\sim\Gamma(k-1)a_N^{k-1}\ell(a_N)$ by (\ref{tilde}).
   Thus, the coalescence probability $c_N$ satisfies
   \[
   c_N\ =\ N\me(W_1^2)\ \sim\
   \frac{N\ell(a_N)}{a_N}\ \sim\ \frac{\ell(a_N)}{\ell^*(a_N)}\ \to\ 0,
   \qquad N\to\infty,
   \]
   and
   \[
   \frac{\Phi_1^{(N)}(k)}{c_N}
   \ =\ \frac{\me(W_1^k)}{\me(W_1^2)}\ \to\  \frac{1}{k-1}
   \ =\ \int_{[0,1]} x^{k-2}\,\Lambda({\rm d}x),
   \qquad k\in\nz\setminus\{1\},
   \]
   where $\Lambda$ denotes the uniform distribution on $[0,1]$.
   To see that simultaneous multiple collisions cannot occur in the limit
   note that $(N)_2\me(W_1W_2)=1-c_N\sim 1$ as $N\to\infty$, or,
   equivalently, $\me(W_1W_2)\sim 1/N^2$ as $N\to\infty$.
   Thus, (\ref{first}) reduces for $j=2$ and $k_1=k_2=1$ to
   \[
   \frac{1}{N^2}\ \sim\ \frac{\varphi_1^2(1/a_N)}{a_N^2}\int_0^{\delta a_N} t\me(e^{-tS_{N-2}/a_N})\,{\rm d}t
   \ \sim\ \Big(\frac{\ell^*(a_N)}{a_N}\Big)^2\int_0^{\delta a_N} t\me(e^{-tS_{N-2}/a_N})\,{\rm d}t,
   \]
   or, equivalently,
   \begin{equation} \label{interim2}
      \int_0^{\delta a_N}t\me(e^{-tS_{N-2}/a_N})\,{\rm d}t\ \sim\
      \Big(\frac{a_N}{N\ell^*(a_N)}\Big)^2\ \sim\ 1,\qquad N\to\infty.
   \end{equation}
   Therefore, for $j=2$ and $k_1,k_2\in\nz\setminus\{1\}$, (\ref{first}) reduces to
   \[
   \me(W_1^{k_1}W_2^{k_2})
   \ \sim\ \frac{\varphi_{k_1}(1/a_N)\varphi_{k_2}(1/a_N)}{\Gamma(k)a_N^k}
   \ \sim\ \frac{\Gamma(k_1-1)\Gamma(k_2-1)}{\Gamma(k)}
           \Big(\frac{\ell(a_N)}{a_N}\Big)^2,
   \]
   where the last asymptotics holds since
   $\varphi_{k_i}(1/a_N)\sim\Gamma(k_i-1)a_N^{k_i-1}\ell(a_N)$ by (\ref{tilde}).
   In particular, $\Phi_2^{(N)}(2,2)=(N)_2\me(W_1^2W_2^2)\sim(N\ell(a_N)/a_N)^2/6
   \sim c_N^2/6$. For $j,k_1,\ldots,k_j\in\nz\setminus\{1\}$ it follows
   from the monotonicity property (\ref{monotone})
   that $\Phi_j^{(N)}(k_1,\ldots,k_j)\le\Phi_2^{(N)}(2,2)=O(c_N^2)$, and,
   therefore,
   $\lim_{N\to\infty}\Phi_2^{(N)}(k_1,\ldots,k_j)/c_N=0$, which shows that
   simultaneous multiple collisions cannot occur in the limit.\\
   To summarize, by \cite[Theorem 2.1]{moehlesagitov}, the model is
   in the domain of attraction of the $\Lambda$-coalescent with
   $\Lambda$ the uniform distribution on $[0,1]$, which is the Bolthausen--Sznitman
   coalescent.\hfill$\Box$
\end{proof}
\begin{remark}
   Suppose that (\ref{maincond}) holds with $\alpha\in(0,1)$. Using the same
   techniques as in the previous proof, it follows for all
   $j\in\nz$ and $k_1,\ldots,k_j\ge 2$ that
   \[
   \Phi_j^{(N)}(k_1,\ldots,k_j)
   \ =\ (N)_j\me(W_1^{k_1}\cdots W_j^{k_j})
   \ \sim\ \frac{\Gamma(j)\Gamma(k_1-1)\cdots\Gamma(k_j-1)}{\Gamma(k)}c_N^j,
   \qquad N\to\infty,
   \]
   where $c_N\sim N\ell(a_N)/a_N\sim\ell(a_N)/\ell^*(a_N)\to 0$ as
   $N\to\infty$. Thanks to the monotonicity property (\ref{monotone})
   this formula is only needed for $j\in\{1,2\}$ in the previous proof.
\end{remark}
We finally turn to the case $\alpha=0$ corresponding to the last
part (vi) of Theorem \ref{main}.
\begin{proof} (of Theorem \ref{main} (vi))
   Let $Q_N$ denote the distribution of $X_2+\cdots+X_N\stackrel{d}{=}S_{N-1}$. For all $p>0$,
   \begin{equation} \label{master}
      \me(W_1^p)
      \ =\ \me(W_1^p1_{\{X_2+\cdots+X_N\le N\}}) + \int_{(N,\infty)}
           \me\bigg(\bigg(\frac{X}{X+x}\bigg)^p\bigg)
           \,Q_N({\rm d}x).
   \end{equation}
   From Lemma \ref{lem1} it follows that there exists $q\in (0,1)$ such that
   $\me(W_1^p1_{\{X_2+\cdots+X_N\le N\}})\le\pr(S_{N-1}\le N)\le q^N$ for all
   sufficiently large $N$. By Lemma \ref{lem2}, $\me((X/(X+x))^p)\sim\ell(x)$
   as $x\to\infty$, which implies that
   \begin{equation} \label{myasy}
      \int_{(N,\infty)}\me\bigg(\bigg(\frac{X}{X+x}\bigg)^p\bigg)
      \,Q_N({\rm d}x)
      \ \sim\ \int_{(N,\infty)}\ell(x)\,Q_N({\rm d}x),
      \qquad N\to\infty.
   \end{equation}
   Note that the integral on the right hand side of (\ref{myasy}) does
   not depend on the parameter $p$. For $p=1$, taking $\me(W_1)=1/N$
   into account, Eq.~(\ref{master}), multiplied by $N$, turns into
   \[
   1\ =\ N\me(W_11_{\{X_2+\cdots+X_N\le N\}}) + N\int_{(N,\infty)}
   \me\bigg(\frac{X}{X+x}\bigg)\,Q_N({\rm d}x).
   \]
   Noting that, for all sufficiently large
   $N$, $N\me(W_11_{\{X_2+\cdots+X_N\le N\}})\le N\pr(S_{N-1}\le N)\le Nq^N\to 0$
   as $N\to\infty$ it follows that $\lim_{N\to\infty}N\int_{(N,\infty)}
   \me(X/(X+x))\,Q_N({\rm d}x)=1$, or, equivalently,
   \[
   \frac{1}{N}
   \ \sim\ \int_{(N,\infty)}\me\bigg(\frac{X}{X+x}\bigg)\,Q_N({\rm d}x)
   \ \sim\ \int_{(N,\infty)}\ell(x)\,Q_N({\rm d}x),\qquad N\to\infty,
   \]
   where the last asymptotics holds by (\ref{myasy}) for $p=1$. Therefore,
   for every $p>0$ the integral in (\ref{myasy}) is asymptotically equal
   to $1/N$ and it follows from (\ref{master}) that $N\me(W_1^p)\to 1$ as
   $N\to\infty$ for all $p>0$. In particular, $c_N=N\me(W_1^2)\to 1$ as
   $N\to\infty$. Moreover, $\Phi_2^{(N)}(2,2)=(N)_2\me(W_1^2W_2^2)\le
   (N)_2\me(W_1W_2)=1-c_N\to 0$ as $N\to\infty$. Thus,
   for all $j,k_1,\ldots,k_j\in\nz\setminus\{1\}$,
   $\Phi_j^{(N)}(k_1,\ldots,k_j)\le\Phi_2^{(N)}(2,2)\to 0$ as $N\to\infty$, which
   shows that simultaneous multiple collisions cannot occur in the limit.
   By \cite[Theorem 2.1]{moehlesagitov}, the model is in the domain of
   attraction of the discrete-time star-shaped coalescent.\hfill$\Box$
\end{proof}
\subsection{Appendix}
For convenience we record the following version of the monotone density theorem.
\begin{lemma} \label{mdt}
   Let $x_0\in (0,\infty]$ and assume that $G:(0,x_0)\to\rz$ has the form
   $G(x)=\int_{(x,x_0)}g(y)\,\lambda({\rm d}y)$ for some measurable function
   $g:(0,x_0)\to\rz$. If $G(x)\sim x^{-\rho}\ell(x)$ as $x\to 0$ for some
   constant $\rho\in[0,\infty)$ and some function $\ell$ slowly varying at
   $0$ and if $g$ is monotone in some right neighborhood of $0$, then
   $\lim_{x\to 0}x^{\rho+1}g(x)/\ell(x)=\rho$.
\end{lemma}
\begin{remark}
   Note that $G'(x)=-g(x)$. The statement of the lemma is hence equivalent
   to $\lim_{x\to 0}xG'(x)/G(x)=-\rho$.
\end{remark}
The following proof of Lemma \ref{mdt} almost exactly coincides with the
proofs known for standard versions of the monotone density theorem (see,
for example, Bingham, Goldie and Teugels
\cite[Theorem 1.7.2]{binghamgoldieteugels} or Feller \cite[p.~446]{feller}.
The proof is provided, since the monotone density theorem in the form of
Lemma \ref{mdt} is heavily used throughout the proofs in Section \ref{proofs}.
\begin{proof} (of Lemma \ref{mdt})
   Suppose first that $g$ is non-increasing in some right neighborhood
   of $0$. If $0<a<b<\infty$, then, for all $x\in (0,x_0/b)$,
   $G(ax)-G(bx)=\int_{(ax,bx]} g(y)\,\lambda({\rm d}y)$
   so, for $x$ small enough,
   \[
   \frac{(b-a)xg(bx)}{x^{-\rho}\ell(x)}
   \ \le\ \frac{G(ax)-G(bx)}{x^{-\rho}\ell(x)}
   \ \le\ \frac{(b-a)xg(ax)}{x^{-\rho}\ell(x)}.
   \]
   The middle fraction is
   \[
   \frac{G(ax)}{(ax)^{-\rho}\ell(ax)} a^{-\rho}\frac{\ell(ax)}{\ell(x)}
   -\frac{G(bx)}{(bx)^{-\rho}\ell(bx)}b^{-\rho}\frac{\ell(bx)}{\ell(x)}
   \ \to\ a^{-\rho}-b^{-\rho},\qquad x\to 0,
   \]
   so the first inequality above yields
   \[
   \limsup_{x\to 0} \frac{g(bx)}{x^{-\rho-1}\ell(x)}
   \ \le\ \frac{a^{-\rho}-b^{-\rho}}{b-a}.
   \]
   Taking $b:=1$ and letting $a\uparrow 1$ gives
   \[
   \limsup_{x\to 0}\frac{g(x)}{x^{-\rho-1}\ell(x)}\ \le\ \lim_{a\to 1}
   \frac{a^{-\rho}-1}{1-a}\ =\ \rho.
   \]
   By a similar treatment of the right inequality with $a:=1$ and
   $b\downarrow 1$ we find that the $\liminf$ is at least $\rho$,
   and the conclusion follows. The argument when $g$ is non-decreasing
   in some right neighborhood of $0$ is similar.\hfill$\Box$
\end{proof}
The following two results are extended versions of Theorem 2 and Theorem 3
of Karamata \cite{karamata3} adapted to our purposes. Lemma \ref{karamatalemma}
provides conditions under which a slowly varying part below an integral can
be moved in front of the integral without changing the asymptotics of the
integral. Corollary \ref{karamatacorollary} is a similar results for the
regularly varying case. The results are slightly more general than those
provided in \cite{karamata3}, since the functions $g_N$ and $f_N$ arising in
the statements are allowed to depend on $N$, which is not the case in the
formulation of \cite{karamata3}.
\begin{lemma} \label{karamatalemma}
   Let $L:(0,\infty)\to(0,\infty)$ be slowly varying at $0$ (or $\infty$),
   let $(x_N)_{N\in\nz}$ be a sequence of positive real numbers satisfying
   $x_N\to 0$ (or $x_N\to\infty$) as $N\to\infty$. Furthermore, let
   $g_N:(0,\infty)\to[0,\infty)$ be nonnegative, integrable functions
   with $0<\int_0^\infty g_N(t)\,{\rm d}t<\infty$ for all $N\in\nz$ and
   such that, for some $a>0$ and some $\eta>0$,
   \[
   \int_0^a t^{-\eta}g_N(t)\,{\rm d}t\ <\ \infty
   \quad\mbox{and}\quad
   \int_a^\infty t^\eta g_N(t)\,{\rm d}t\ <\ \infty
   \]
   for all $N\in\nz$. Then, as $N\to\infty$,
   \[
   \int_0^\infty L(x_Nt)g_N(t)\,{\rm d}t
   \ \sim\ L(x_N)\int_0^\infty g_N(t)\,{\rm d}t.
   \]
\end{lemma}
\begin{proof}
   Define $P(x):=x^\eta L(x)$ and $Q(x):=x^{-\eta}L(x)$, $x>0$. Note that
   $P$ is regularly varying with index $\eta$ and $Q$ is regularly
   varying with index $-\eta$. By \cite[Theorem 1.5.2]{binghamgoldieteugels},
   $P(x_Nt)/P(x_N)\to t^\eta$ as $N\to\infty$ uniformly in $t\in(0,a]$ and
   $Q(x_Nt)/Q(x_N)\to t^{-\eta}$ as $N\to\infty$ uniformly in $t\in [a,\infty)$.
   Thus, for every $\varepsilon>0$ there exists $N_0=N_0(\varepsilon)\in\nz$
   such that, for all $N\in\nz$ with $N>N_0$,
   \[
   P(x_N)(1-\varepsilon)
   \ \le\ t^{-\eta}P(x_Nt)
   \ \le\ P(x_N)(1+\varepsilon)\quad\mbox{for all $t\in(0,a]$}
   \]
   and
   \[
   Q(x_N)(1-\varepsilon)
   \ \le\ t^\eta Q(x_Nt)
   \ \le\ Q(x_N)(1+\varepsilon)\quad\mbox{for all $t\in[a,\infty)$}.
   \]
   For all $N\in\nz$ with $N>N_0$ it follows that
   \begin{eqnarray*}
     \int_0^\infty L(x_Nt)g_N(t)\,{\rm d}t
     & = & x_N^{-\eta}\int_0^a t^{-\eta}P(x_Nt)g_N(t)\,{\rm d}t
           + x_N^\eta\int_a^\infty t^\eta Q(x_Nt)g_N(t)\,{\rm d}t\\
     & \le & x_N^{-\eta}P(x_N)(1+\varepsilon)\int_0^a g_N(t)\,{\rm d}t
             + x_N^\eta Q(x_N)(1+\varepsilon)\int_a^\infty g_N(t)\,{\rm d}t \\
     & = & (1+\varepsilon)L(x_N)\int_0^\infty g_N(t)\,{\rm d}t
   \end{eqnarray*}
   and, analogously,
   $\int_0^\infty L(x_Nt)g_N(t)\,{\rm d}t\ge(1-\varepsilon)L(x_N)\int_0^\infty g_N(t)\,{\rm d}t$.
   \hfill$\Box$
\end{proof}
\begin{corollary} \label{karamatacorollary}
   Let $\varphi:(0,\infty)\to (0,\infty)$ be regularly varying at $0$
   (or $\infty$) with index $\gamma\in\rz$ and let $(x_N)_{N\in\nz}$ be a
   sequence of positive real numbers satisfying $x_N\to 0$ (or $x_N\to\infty$)
   as $N\to\infty$. Furthermore, let $f_N:(0,\infty)\to [0,\infty)$, $N\in\nz$,
   be functions such that $0<\int_0^\infty t^\eta f_N(t)\,{\rm d}t<\infty$ for
   all $N\in\nz$ and all $\eta$ in some neighborhood of $\gamma$, i.e. for all
   $\eta\in(\gamma-\varepsilon,\gamma+\varepsilon)$ for some $\varepsilon>0$.
   Then, as $N\to\infty$,
   \[
   \int_0^\infty \varphi(x_Nt)f_N(t)\,{\rm d}t
   \ \sim\ \varphi(x_N)\int_0^\infty t^\gamma f_N(t)\,{\rm d}t.
   \]
\end{corollary}
\begin{proof}
   Define $L(x):=x^{-\gamma}\varphi(x)$ for $x>0$, and $g_N(t):=t^\gamma f_N(t)$
   for $t>0$. Choose $\eta:=\varepsilon/2>0$. Then, for any $a>0$,
   $\int_0^a t^{-\eta}g_N(t)\,{\rm d}t=\int_0^a
   t^{\gamma-\eta}f_N(t)\,{\rm d}t\le\int_0^\infty t^{\gamma-\eta}f_N(t)\,{\rm d}t<\infty$
   by assumption and as well $\int_a^\infty t^\eta g_N(t)\,{\rm d}t=
   \int_a^\infty t^{\gamma+\eta}f_N(t)\,{\rm d}t\le\int_0^\infty t^{\gamma+\eta}f_N(t)\,{\rm d}t<\infty$
   by assumption. Thus, Lemma \ref{karamatalemma} is applicable, which
   yields the result.\hfill$\Box$
\end{proof}
\begin{acknowledgement}
   The authors thank two anonymous referees for their useful reports
   leading to a significant improvement of the proofs of parts (iv) and
   (v) of Theorem \ref{main}.
\end{acknowledgement}

\end{document}